\documentclass[11pt]{article}
\usepackage[utf8]{inputenc}
\usepackage{geometry}
\usepackage{amsmath,amsthm,amssymb,chngcntr,setspace,titlesec}
\usepackage[all]{xy}

\newtheorem{thm}{Theorem}[section]
\newtheorem{lemma}[thm]{Lemma}

\newtheorem{prop}[thm]{Proposition}
\newtheorem{cor}[thm]{Corollary}
\newtheorem{defn}[thm]{Definition}
\newtheorem{examp}[thm]{Example}

\newtheorem{rmk}[thm]{Remark}

\title{Weil-\'Etale Cohomology And Special Values Of L-functions At Zero}
\author{Minh-Hoang Tran
}
\begin{document}
\date{}
\maketitle

\begin{abstract}
We construct the Weil-\'etale cohomology and Euler characteristics for a subclass of the class of $\mathbb{Z}$-constructible sheaves on the spectrum of the ring of integers of a totally imaginary number field. Then we show that the special value of an Artin L-function of toric type at zero is given by the Weil-\'etale Euler characteristic of an appropriate  $\mathbb{Z}$-constructible sheaf up to signs. As applications of our result, we will derive some classical formulas of special values of L-functions of tori and their class numbers.
\end{abstract}
\section{Introduction}

Through out this paper, we fix the following notations. $K$ is a totally imaginary number field with ring of integers $O_K$ and Galois group $G_K$. Let $X=Spec(O_K)$ and $j: Spec(K)\to X$. By a discrete $G_K$-module we mean a finitely generated abelian group with a continuous $G_K$ action. There are two main aims of this paper.
\begin{enumerate}
\item To construct the Weil-\'etale cohomology $H^n_W(X,\mathcal{F})$ and Euler characteristic $\chi(\mathcal{F})$ for any strongly-$\mathbb{Z}$-constructible sheaf $\mathcal{F}$ on $X$ (see definition $\ref{strong_const}$).
\item Let $N$ be a discrete $G_K$-module and  let $L(N,s)$ be the Artin L-function associated with the representation $N\otimes_{\mathbb{Z}}\mathbb{C}$ of $G_K$. We will show that $j_{*}N$ is strongly-$\mathbb{Z}$-constructible and  $L^{*}(N,0)=\pm \chi(j_{*}N)$.
\end{enumerate}
  
The fact that $L^{*}(N,0)$ is related to the Euler characteristic of $j_{*}N$ was established by Bienenfeld and Lichtenbaum and our proof is based on the techniques they developed in $\cite{Lic75}$ and $\cite{BL}$. However, their Euler characteristic is different from the Weil-\'etale Euler characteristic constructed in this paper. The Weil-\'etale cohomology in this paper is not the same as the Weil-\'etale cohomology constructed by Lichtenbaum in $\cite{Lic09a}$ but rather is based on his ideas in $\cite{Lic09b}$ and $\cite{Lic14}$.  As applications, let $T$ be an algebraic torus over $K$ with character group $\hat{T}$, we obtain a formula for $L^{*}(\hat{T},0)$ which is similar to Ono's Tamagawa number formula for tori $\cite{Ono63}$ and use it to derive some formulas of Ono $\cite{Ono87}$ and Katayama $\cite{Kat91}$ for the class numbers of tori.

The structure of the paper is as follows. We construct the Weil-\'etale cohomology for $\mathbb{Z}$-constructible sheaves in section 2. In sections 3 and 4, we define the class of strongly-$\mathbb{Z}$-constructible sheaves and construct their Weil-\'etale Euler characteristics. Our main result  $L^{*}(N,0)=\pm \chi(j_{*}N)$ is proved at the end of section 4. Section 5 is for applications and examples. Finally, we have an appendix containing the results about determinants of exact sequences and orders of torsion groups used in this paper.

\textbf{Acknowledgment:} This paper is part of my PhD thesis written at Brown University. I would like to thank my advisor Professor Stephen Lichtenbaum for his guidance and encouragement. Part of this work was written when I was a member of the SFB Higher Invariant Research Group at University of Regensburg. I would like to thank Professor Guido Kings for his support and my friend Yigeng Zhao for many helpful conversations.
\section{The Weil-\'Etale Cohomology Of $\mathbb{Z}$-Constructible Sheaves}
\subsection{The Weil-\'Etale Complexes}
In this section, we define the Weil-\'etale complex for $\mathbb{Z}$-constructible sheaves following the ideas of Lichtenbaum $\cite{Lic14}$. First, we recall the definition of $\mathbb{Z}$-constructible sheaves from $\cite[\mbox{page 146}]{Mil06}$.
\begin{defn}\label{defn_construct}
	A sheaf $\mathcal{F}$ on $X$ is $\mathbb{Z}$-constructible if 
		\begin{enumerate}
			\item there exists an open dense subscheme $U$ of $X$ and a finite \'etale covering 
	$ U' \to U$ such that the restriction of $\mathcal{F}$ to $U'$ is a constant sheaf defined by a finitely generated abelian group, 
			\item for any point $p$ outside $U$, the stalk $\mathcal{F}_{\bar{p}}$ is a finitely generated abelian group.
		\end{enumerate}
We say that $\mathcal{F}$ is constructible if in the definition above the restriction of $\mathcal{F}$ to $U'$ is a  constant sheaf defined by a finite abelian group and for any point $p$ outside $U$, the stalk $\mathcal{F}_{\bar{p}}$ is finite.
\end{defn}

\begin{examp}
\begin{enumerate}
\item Any constant sheaf defined by a finitely-generated abelian groups.
\item Let $M$ be a discrete $G_K$-module, then $j_{*}M$ is a $\mathbb{Z}$-constructible sheaf. Furthermore, if $M$ is finite then $j_{*}M$ is constructible.
\end{enumerate}
 \end{examp}
 \begin{defn}\label{defn_negligible}
We say $\mathcal{F}$ is a negligible sheaf on $X$ if there exists a finite set $S$ of closed points of $X$ such that $ \mathcal{F}=\prod_{p \in S}(i_{p})_{*}M_p$ where $M_p$ is a finite discrete $ \hat{\mathbb{Z}}$-module and $i_p$ is the map $p \to X$. Note that negligible sheaves are constructible.
\end{defn}
\begin{defn}\label{Weil_etale_defn}
Let  $\mathcal{F}$ be $\mathbb{Z}$-constructible, the Weil-\'etale complex is defined as
\[ R\Gamma_W(X,\mathcal{F}) := 
\tau_{\leq 1}R\Gamma_{et}(X,\mathcal{F}) \oplus
 \tau_{\geq 2}R\mathrm{Hom}_{\mathbb{Z}}(R\mathrm{Hom}_{X}(\mathcal{F},\mathbb{G}_m),\mathbb{Z}[-2]) \]
where $\tau_{\leq n}$ and $\tau_{\geq n}$ are the truncation functors defined in $\cite[\mbox{1.2.7}]{Weibel94}$. This is an object in the derived category of abelian groups. The Weil-\'etale cohomology are defined as $H^n_W(X,\mathcal{F}):=h^n(R\Gamma_W(X,\mathcal{F}))$.
\end{defn}

\begin{prop}\label{Weil_F}
The Weil-\'etale cohomology of $\mathcal{F}$ satisfy
\begin{equation}
H^n_W(X,\mathcal{F})= 
       \left\{
				\begin{array}{ll}
					H^n_{et}(X,\mathcal{F})&  \mbox{$n=0,1$}\\
					\mathrm{Hom}_X(\mathcal{F},\mathbb{G}_m)_{tor}^D & \mbox{$n=3$}  \\
					0 & \mbox{$ n > 3$}.
				\end{array}
			   \right.
\end{equation}
\[ 0 \to Ext^{1}_{X}(\mathcal{F},\mathbb{G}_m)_{tor}^D \to H^{2}_W(X,\mathcal{F}) \to Hom_{X}(\mathcal{F},\mathbb{G}_m)^{*} \to 0. \]
Note that  $A^D:=Hom_{\mathbb{Z}}(A,\mathbb{Q}/\mathbb{Z})$ and $A^{*}:=Hom_{\mathbb{Z}}(A,\mathbb{Z})$.
\end{prop}
\begin{proof}
From the definition of $R\Gamma_W(X,\mathcal{F})$, we have
\[ H^n_W(X,\mathcal{F})= 
       \left\{
				\begin{array}{ll}
					H^n_{et}(X,\mathcal{F})&  \mbox{$n=0,1$}\\
h^n(R\mathrm{Hom}_{\mathbb{Z}}(R\mathrm{Hom}_{X}(\mathcal{F},\mathbb{G}_m),\mathbb{Z}[-2]))& \mbox{$n \geq 2$.}  \\
				\end{array}
			   \right.
\]
For $n\geq 2$, from $\cite[\mbox{exercise 3.6.1}]{Weibel94}$, there is an exact sequence 
\[ 0 \to Ext^{3-n}_{X}(\mathcal{F},\mathbb{G}_m)_{tor}^D \to H^{n}_W(X,\mathcal{F}) \to Ext^{2-n}_{X}(\mathcal{F},\mathbb{G}_m)^{*} \to 0 .\]
Hence, $H^{n}_W(X,\mathcal{F})=0$ for $n\geq 4$
and $H^{3}_W(X,\mathcal{F}) \simeq Hom_{X}(\mathcal{F},\mathbb{G}_m)_{tor}^{D}$. 
\end{proof}
To compute the Weil-\'etale cohomology of $\mathbb{Z}$, we need the following result.
\begin{thm}\label{et_ZGm}
Let $K$ be a totally imaginary number field. Then
\begin{eqnarray}
H^{n}_{et}(X,\mathbb{Z}) = 
              \left\{
				\begin{array}{ll}
					\mathbb{Z} &  \mbox{$n=0$}\\
					0 &  \mbox{$n=1$ }\\
                                 Pic(O_K)^D & \mbox{$n=2$} \\
 					(O_K^{*})^D & \mbox{$n=3$}  \\
 					0 & \mbox{$n>3$}

				\end{array}
			   \right.
			 & \mbox{ and } &
	  H^{n}_{et}(X,\mathbb{G}_m) = 
                                   \left\{
				\begin{array}{ll}
					O_{K}^{*} &  \mbox{$n=0$}\\
					Pic(O_K) &  \mbox{$n=1$ }\\
                                0 & \mbox{$n=2$} \\
 					\mathbb{Q}/\mathbb{Z} & \mbox{$n=3$} \\
					0 & \mbox{$n>3$}.
				\end{array}
			   \right.
	            \end{eqnarray}

\end{thm}
\begin{proof}
For $H^n_{et}(X,\mathbb{G}_m)$ see $\cite[\mbox{II.2.1}]{Mil06}$. For $H^n_{et}(X,\mathbb{Z})$, we can apply the Artin-Verdier duality $\cite[\mbox{II.3.1}]{Mil06}$.
\end{proof}

\begin{prop}\label{Weil_ZGm}
The Weil-\'etale cohomology of  $\mathbb{Z}$ is given by
\begin{eqnarray}
H^{n}_{W}(X,\mathbb{Z}) = 
         \left\{
				\begin{array}{ll}
					\mathbb{Z} &  \mbox{$n=0$}\\
					0 &  \mbox{$n=1$ }\\
 					(\mu_K)^D & \mbox{$n=3$} \\
 					0 & \mbox{$n>3$}.
				\end{array}
			   \right.
\end{eqnarray}
\[ 0 \to  Pic(O_K)^D \to H^{2}_{W}(X,\mathbb{Z}) \to Hom_{\mathbb{Z}}(O_K^{*},\mathbb{Z}) \to 0. \]
\end{prop}
\begin{proof}
This is an application of proposition $\ref{Weil_F}$ and theorem $\ref{et_ZGm}$ for $\mathcal{F}=\mathbb{Z}$.
\end{proof}
For each place $v$ of $K$, let $K_v$ be the completion of $K$ at $v$ and $j_v$ be the map $Spec(K_v) \to X$.
\begin{defn}\label{defn_FBFDB}
For a sheaf $\mathcal{F}$ on $X$, we define the Betti sheaf $\mathcal{F}_B$ to be
\[ \mathcal{F}_B:=\prod_{v \in S_{\infty}} (j_v)_{*}(j_v)^{*}\mathcal{F}. \]
Note that $\mathcal{F} \mapsto \mathcal{F}_B$ is an exact functor. Moreover, there is a natural map $\mathcal{F}\to \mathcal{F}_B$ obtained by taking the direct sum over all infinite place $v$ of $K$ of the map $\mathcal{F} \to (j_v)_{*}(j_v)^{*}\mathcal{F}$ .
\end{defn} 
\begin{lemma}\label{et_FB}
For each infinite place $v$ of $K$, let $\mathcal{F}_{K_v}$ be the stalk of $\mathcal{F}$ at $\bar{K_v}$. Then 
\begin{eqnarray*}
H^{n}_{et}(X,\mathcal{F}_B) = 
                                   \left\{
				\begin{array}{ll}
					 \prod_{v \in S_{\infty}} \mathcal{F}_{K_v}&  \mbox{$n=0$}\\
					 0 &  \mbox{$n \neq 0$}.\\
				\end{array}
			   \right.
\end{eqnarray*}
\end{lemma}
\begin{proof}
 This follows from the fact that $(j_v)_{*}$ is an exact functor and $G_{K_v}=0$ for $v\in S_{\infty}$. 
\end{proof} 
\subsection{The Regulator Pairing}
We want to define a pairing for every \'etale sheaf on $X$ such that it generalizes the construction of the regulator of a number field when the sheaf is $\mathbb{Z}$. 
There is a natural map 
\[ \Lambda_K : \frac{H^0_{et}(X,(\mathbb{G}_m)_B)}{H^0_{et}(X,\mathbb{G}_m)}=\frac{\prod_{v \in S_{\infty}} K_v^{*}}{O_K^{*}} \to \mathbb{R} \] 
\[ \textbf{u}=(u_1,...,u_v) \mapsto \sum_{v \in S_{\infty}} \log|u_v|_v. \]
Note that by the product formula, $\Lambda_K$ is well-defined.
\begin{defn}\label{defn_pairing} 
Let $\mathcal{F}$ be an \'etale sheaf on $X$. The regulator pairing for $\mathcal{F}$ is defined as  
\begin{equation}\label{eqn_pairing}
 \langle \cdot , \cdot \rangle_{\mathcal{F}} :  \frac{H^0_{et}(X,\mathcal{F}_B)}{H^0_{et}(X,\mathcal{F})} \times Hom_{X}(\mathcal{F},\mathbb{G}_m) \to \mathbb{R} 
\end{equation}
let $\alpha$ and $\phi$ be elements of $\frac{H^0_{et}(X,\mathcal{F}_B)}{H^0_{et}(X,\mathcal{F})}$ and $Hom_{X}(\mathcal{F},\mathbb{G}_m)$. By functoriality, $\phi$ induces a map 
\[  \phi_B : \frac{H^0_{et}(X,\mathcal{F}_B)}{H^0_{et}(X,\mathcal{F})} \to  
\frac{H^0_{et}(X,(\mathbb{G}_m)_B)}{H^0_{et}(X,\mathbb{G}_m)} 
= \frac{\prod_{v \in S_{\infty}} K_v^{*}}{O_K^{*}}. \]
Define 
$ \langle \alpha,\phi\rangle_{\mathcal{F}} := \Lambda_K(\phi_B(\alpha)) = 
\sum_{v \in S_{\infty}} \log|\phi_B(\alpha)_v|_v .$
\end{defn}

\begin{defn}\label{defn_regulator}
Suppose the pairing $\langle,\rangle_{\mathcal{F}}$ is non-degenerate modulo torsion. Choose bases $\{v_i\}$ and $\{u_j\}$ for the torsion free quotient groups of
$H^0_{et}(X,\mathcal{F}_B)/H^0_{et}(X,\mathcal{F})$ and $Hom_{X}(\mathcal{F},\mathbb{G}_m)$ respectively. Define $R(\mathcal{F}):=|\det(\langle v_i,u_j\rangle_{\mathcal{F}})|$ to be {the regulator of $\mathcal{F}$}. This definition does not depend on the choice of bases.
\end{defn}

\begin{examp}\label{ex_pairing}
\begin{enumerate}
\item If $\mathcal{F}$ is constructible then $Hom_{X}(\mathcal{F},\mathbb{G}_m)$ and $H^0_{et}(X,\mathcal{F}_B)$ are finite groups. Thus, the pairing is non-degenerate modulo torsion and $R(\mathcal{F})=1$.
\item Consider the constant sheaf $\mathbb{Z}$ on $X$, the regulator pairing in this case is 
\[ \frac{(\prod_{v \in S_{\infty}}\mathbb{Z})}{\mathbb{Z}} \times O_K^{*} \to \mathbb{R} \]
Let $\{u_1,...u_{r_1+r_2-1}\}$ be a $\mathbb{Z}$-basis for $O_K^{*}/\mu_K$. Let $\{v_1,..v_{r_1+r_2-1}\}$ be any $r_1+r_2-1$ embeddings of $K$ into $\mathbb{C}$ considered as a basis for 
 $\prod_{v \in S_{\infty}}\mathbb{Z}/\mathbb{Z}$. Then $\langle v_i,u_j\rangle=\log|u_j|_{v_i}$. The determinant of the matrix $(\log|u_j|_{v_i})$ is the regulator $R$ of the number field $K$. 
Hence $R(\mathbb{Z})=R$.
\item Let $T$ be an algebraic torus over a totally imaginary number field $K$. Then the regulator 
$R({j_{*}\hat{T}})$ is the same as the regulator $R_T$ of the torus $T$ defined in $\cite{Ono61}$.
\end{enumerate}
\end{examp}

\begin{lemma}\label{pairing_functor}
Let $f : \mathcal{F} \to \mathcal{G}$ be a morphism of sheaves. Then the following diagram commutes
\begin{equation}\label{pairing_functor_1}
 \xymatrixrowsep{0.2in}\xymatrix{ \frac{H^0_{et}(X,\mathcal{F}_{B})}{H^0_{et}(X,\mathcal{F})} \ar[d]^{f_B} & \times & Hom_{X}(\mathcal{F},\mathbb{G}_m)  \ar[r] & \mathbb{R} \ar[d]^{id} \\
\frac{H^0_{et}(X,\mathcal{G}_{B})}{H^0_{et}(X,\mathcal{G})} & \times & Hom_{X}(\mathcal{G},\mathbb{G}_m) \ar[u]^{f^{*}} \ar[r] & \mathbb{R} }
\end{equation}
\end{lemma}
\begin{proof}
Let $\alpha$ and $\phi$ be elements of $H^0_{et}(X,\mathcal{F}_{B})$ and $Hom_{X}(\mathcal{G},\mathbb{G}_m)$ respectively. Then
\begin{eqnarray*}
\langle f_B(\alpha),\phi \rangle = \Lambda_K(\phi_B(f_B(\alpha)))
= \Lambda_K((f^{*}\phi)_B(\alpha)) = \langle \alpha,f^{*}\phi \rangle.
\end{eqnarray*}
\end{proof}

\section{Strongly-$\mathbb{Z}$-Constructible Sheaves}
\label{chapter_strongly_Z}
 \subsection{Definitions And Examples}
\begin{defn}\label{strong_const}
An \'etale sheaf $\mathcal{F}$ on $X=Spec(O_K)$ is called strongly-$\mathbb{Z}$-constructible if it satisfies the following conditions :
\begin{enumerate}
\item It is $\mathbb{Z}$-constructible.
\item The map $ H^0_{et}(X,\mathcal{F}) \to H^0_{et}(X,\mathcal{F}_B) $ has finite kernel.
\item $H^1_{et}(X,\mathcal{F})$ and $H^2_{et}(X,\mathcal{F})$ are finite abelian groups.
\item The regulator pairing $(\ref{defn_pairing})$ is non-degenerate modulo torsion.
\end{enumerate}
\end{defn}

\begin{rmk}\label{Artin_Verdier_strongly}
Our definition is modeled after the definition of quasi-constructible sheaves of Bienenfeld and Lichtenbaum (cf. $\cite[\mbox{section 4}]{BL}$). 
\end{rmk}

\begin{examp}
\begin{enumerate}
\item Constant sheaves defined by finitely generated abelian groups.
\item Let $p$ be a closed point of $X$ and  $i : p \to X$ be the natural map. Let M be a finite $\hat{\mathbb{Z}}$-module. Then $i_{*}M$ is strongly-$\mathbb{Z}$-constructible. A non-example would be $i_{*}\mathbb{Z}$. Indeed, 
$H^2_{et}(X,i_{*}\mathbb{Z}) \simeq H^2(\hat{\mathbb{Z}},\mathbb{Z})\simeq \mathbb{Q}/\mathbb{Z}$ which is infinite. 
\item Constructible sheaves.
\item Let $M$ be a discrete $G_K$-module. Then $j_{*}M$ is strongly-$\mathbb{Z}$-constructible (see proposition $\ref{jM_strong_constr}$).
\end{enumerate}
\end{examp}

The following proposition is a direct consequence of $\ref{Weil_F}$ and the Artin-Verdier duality.
\begin{prop}
Let $\mathcal{F}$ be a strongly-$\mathbb{Z}$-constructible sheaf on $X$. Then 
 \[H^{n}_{W}(X,\mathcal{F}) = 
       \left\{
				\begin{array}{ll}
					H^n_{et}(X,\mathcal{F}) & \mbox{$n=0,1$} \\
					Hom_X(\mathcal{F},\mathbb{G}_m)_{tor}^D &  \mbox{$n=3$}\\
					0 &  \mbox{$n \not\in \{0,1,2,3\}$.}
				\end{array}
			   \right.
	            \]           
 \[ 0 \to H^2_{et}(X,\mathcal{F}) \to H^2_{W}(X,\mathcal{F}) \to Hom_X(\mathcal{F},\mathbb{G}_m)^{*} \to 0.\]
 In particular, if $\mathcal{F}$ is a constructible sheaf, then
$H^n_{W}(X,\mathcal{F})=H^n_{et}(X,\mathcal{F})$ for all $n$. 
\end{prop}

\begin{prop}\label{long_exact_Weil}
Suppose we have an exact sequence of strongly-$\mathbb{Z}$-constructible sheaves
\begin{equation}\label{short_exact_etale}
 0 \to \mathcal{F}_1 \to \mathcal{F}_2 \to \mathcal{F}_3 \to 0 
\end{equation}
Then we have long exact sequences of Weil-\'etale cohomology
\begin{equation}\label{long_Weil_a}
 0 \to H^0_{W}(X,\mathcal{F}_1) \to H^0_{W}(X,\mathcal{F}_2)  \to ... \to H^3_{W}(X,\mathcal{F}_2) \to H^3_{W}(X,\mathcal{F}_3) \to  0 
\end{equation}
\end{prop}

\begin{proof}
From $\cite[\mbox{exercise 3.6.1}]{Weibel94}$, for $i\in \{1,2,3\}$ we have 
\[ 0 \to Ext^2_X(\mathcal{F}_i,\mathbb{G}_m)_{tor}^D \to 
h^1(R\mathrm{Hom}_{\mathbb{Z}}(R\mathrm{Hom}_{X}(\mathcal{F}_i,\mathbb{G}_m),\mathbb{Z}[-2]))
\to Ext^1_X(\mathcal{F}_i,\mathbb{G}_m)^{*} \to 0. \]
As $\mathcal{F}_i$ is strongly-$\mathbb{Z}$-constructible, $Ext^1_X(\mathcal{F}_i,\mathbb{G}_m)$ is finite and 
$Ext^2_X(\mathcal{F}_i,\mathbb{G}_m)_{tor}^D \simeq H^1_{et}(X,\mathcal{F}_i)$. Hence $ h^1(R\mathrm{Hom}_{\mathbb{Z}}(R\mathrm{Hom}_{X}(\mathcal{F}_i,\mathbb{G}_m),\mathbb{Z}[-2])) 
\simeq H^1_{et}(X,\mathcal{F}_i)$ .

As $R\mathcal{H}om(-,\mathbb{G}_m[-1])$, $R\Gamma_{et}(X,-)$ and $R\mathrm{Hom}_{\mathbb{Z}}(-,\mathbb{Z}[-3])$ are exact functors, we have the following distinguished triangle 
\begin{multline}\label{long_exact_Weil_seq1}
 R\mathrm{Hom}_{\mathbb{Z}}(R\Gamma_{et}(X,\mathcal{F}_1^D),\mathbb{Z}[-3]) \to 
R\mathrm{Hom}_{\mathbb{Z}}(R\Gamma_{et}(X,\mathcal{F}_2^D),\mathbb{Z}[-3]) \to \\
\to R\mathrm{Hom}_{\mathbb{Z}}(R\Gamma_{et}(X,\mathcal{F}_3^D),\mathbb{Z}[-3]) \to 
R\mathrm{Hom}_{\mathbb{Z}}(R\Gamma_{et}(X,\mathcal{F}_1^D),\mathbb{Z}[-3])[1].
\end{multline}
The long exact sequence of cohomology corresponding to ($\ref{long_exact_Weil_seq1}$) yields 
\begin{multline}\label{long_exact_Weil_seq2}
H^1_{et}(X,\mathcal{F}_1) \to H^1_{et}(X,\mathcal{F}_2) \to H^1_{et}(X,\mathcal{F}_3) \to 
H^2_W(X,\mathcal{F}_1) \to H^2_W(X,\mathcal{F}_2) \to \\ 
\to H^2_W(X,\mathcal{F}_3) \to 
H^3_W(X,\mathcal{F}_1) \to H^3_W(X,\mathcal{F}_2) \to H^3_W(X,\mathcal{F}_3) \to 0.
\end{multline}
Combining ($\ref{long_exact_Weil_seq2}$) with the first 6-term of the long exact sequence of \'etale cohomology groups corresponding to ($\ref{short_exact_etale}$), we have ($\ref{long_Weil_a}$).
\end{proof}

\subsection{Main Properties}
We study the main properties of strongly-$\mathbb{Z}$-constructible sheaves in this section. 
\begin{prop}\label{abel_constr}
Suppose we have an exact sequence of \'etale sheaves on $X$
\begin{equation}\label{abel_constr_seq1}
 0 \to \mathcal{F}_1 \to \mathcal{F}_2 \to \mathcal{F}_3 \to 0 
\end{equation}
where $\mathcal{F}_3$ is constructible. Then $\mathcal{F}_1$ is strongly-$\mathbb{Z}$-constructible if and only if 
$\mathcal{F}_2$ is strongly-$\mathbb{Z}$-constructible.
\end{prop}

\begin{proof}
Since the category of $\mathbb{Z}$-constructible sheaves is an abelian category $\cite[\mbox{page 146}]{Mil06}$, condition 1 of $\ref{strong_const}$ holds for $\mathcal{F}_1$ if and only if it holds for $\mathcal{F}_2$.

Since $H^n_{et}(X,\mathcal{F}_3)$ is finite for all $n$, $H^n_{et}(X,\mathcal{F}_1)$ and $H^n_{et}(X,\mathcal{F}_2)$ differ only by finite groups. As $H^0_{et}(X,\mathcal{F}_{3,B})$ and $H^0_{et}(X,\mathcal{F}_3)$ are finite,
$ \left({H^0_{et}(X,\mathcal{F}_{1,B})}/{H^0_{et}(X,\mathcal{F}_1)}\right)_{\mathbb{R}} \simeq 
\left({H^0_{et}(X,\mathcal{F}_{2,B})}/{H^0_{et}(X,\mathcal{F}_2)}\right)_{\mathbb{R}} $.
Hence, condition 2 and 3 of $\ref{strong_const}$ holds for $\mathcal{F}_1$ if and only if they hold for $\mathcal{F}_2$.

Finally, from lemma $\ref{pairing_functor}$, there is a commutative diagram 
\[ \xymatrixrowsep{0.2in}\xymatrix{ \left(\frac{H^0_{et}(X,\mathcal{F}_{1,B})}{H^0_{et}(X,\mathcal{F}_1)}\right)_{\mathbb{R}} \ar[d]^{\simeq} & \times & Hom_{X}(\mathcal{F}_1,\mathbb{G}_m)_{\mathbb{R}}  \ar[r] & \mathbb{R} \ar[d]^{id} \\
\left(\frac{H^0_{et}(X,\mathcal{F}_{2,B})}{H^0_{et}(X,\mathcal{F}_2)}\right)_{\mathbb{R}} & \times & Hom_{X}(\mathcal{F}_2,\mathbb{G}_m)_{\mathbb{R}} \ar[u]^{\simeq} \ar[r] & \mathbb{R} }
\]
As a result, condition 4 of $\ref{strong_const}$ holds for $\mathcal{F}_1$ if and only if it holds for $\mathcal{F}_2$.
\end{proof}
Next we want to show that strongly-$\mathbb{Z}$-constructible sheaves are stable under push-forward by a finite morphism.
Let $L/K$ be a finite Galois extension of totally imaginary number fields. Let $\pi : Spec(L) \to Spec(K)$ and $ \pi': Y=Spec(O_L) \to X=Spec(O_K)$ be the natural finite morphisms. We write $S_{L,\infty}$ and $S_{K,\infty}$ for the set of infinite places of $L$ and $K$ respectively. 
\begin{lemma}\label{Mackey}
Let $v$ be a place of $K$. Then $j_{v}^{*}\pi'_{*}\mathcal{F}\simeq \prod_{w|v}(\pi_w)_{*}j_{w}^{*}\mathcal{F}$ where $j_{w} : Spec(L_w) \to Y$ and $\pi_{w} : Spec(L_w) \to Spec(K_v)$ are the natural maps. 
\end{lemma}
\begin{proof}
We have the commutative diagram
\[\xymatrixrowsep{0.2in}
\xymatrix @C=1in{\prod_{w|v}Spec(L_w) \ar[d]^{\prod_{w|v}\pi_w}  \ar[r]^{\prod_{w|v}j_w} &  Spec(O_L)  \ar[d]^{\pi'}  \\
                     Spec(K_v)   \ar[r]^{j_v} &  Spec(O_K)  }
                     \]
Since  $ \pi'_{*}(j_w)_{*}=(j_v)_{*}(\pi_w)_{*}$ for $w|v$,
\begin{eqnarray*}
Hom_{K_v}(j_{v}^{*}\pi'_{*}\mathcal{F},\prod_{w|v}(\pi_w)_{*}j_{w}^{*}\mathcal{F})
 = Hom_{X}(\pi'_{*}\mathcal{F},\prod_{w|v}(j_{v})_{*}(\pi_w)_{*}j_{w}^{*}\mathcal{F}) 
= Hom_{X}(\pi'_{*}\mathcal{F},\pi'_{*}\prod_{w|v}(j_{w})_{*}j_{w}^{*}\mathcal{F}) .
\end{eqnarray*}
Thus, the adjoint map $\mathcal{F} \to \prod_{w|v}(j_{w})_{*}j_{w}^{*}\mathcal{F}$ induces a canonical map $j_{v}^{*}\pi'_{*}\mathcal{F} \to \prod_{w|v}(\pi_w)_{*}j_{w}^{*}\mathcal{F}$. Let $\eta_K=Spec(K)$, $\eta_v=Spec(K_v)$ and similarly for $\eta_L$ and $\eta_w$. Then
\[ (j_{v}^{*}\pi_{*}\mathcal{F})_{\eta_v}=\mathcal{F}_{\eta_L}^{[L:K]}=\left(\prod_{w|v}(\pi_w)_{*}j_{w}^{*}\mathcal{F}\right)_{\eta_v}. \]
Therefore, $j_{v}^{*}\pi'_{*}\mathcal{F}\simeq \prod_{w|v}(\pi_w)_{*}j_{w}^{*}\mathcal{F}$. 
\end{proof}

\begin{lemma}\label{pre_finite_inv}
Let $\mathcal{F}$ be a $\mathbb{Z}$-constructible sheaf on $Y$. 
Then the following hold 
\begin{enumerate}
\item The norm map induces a natural isomorphism $Nm: Ext_{Y}^{n}(\mathcal{F},\mathbb{G}_m) \to Ext_{X}^{n}(\pi'_{*}\mathcal{F},\mathbb{G}_m)$. 
\item There is a natural isomorphism $H^{n}_{W}(X,\pi'_{*}\mathcal{F}) \simeq H^{n}_{W}(Y,\mathcal{F})$.
\item The sheaf $(\pi'_{*}\mathcal{F})_B$ is isomorphic to $\pi'_{*}(\mathcal{F}_B)$. In particular, $ H^0_{et}(X,(\pi'_{*}\mathcal{F})_B) \simeq H^0_{et}(Y,\mathcal{F}_B)$.
\end{enumerate}
\end{lemma}
\begin{proof}
\begin{enumerate}
\item
We begin by describing the map $Nm$. The norm map $N_{L/K}$ induces a morphism of sheaves $N_{L/K}:\pi'_{*}\mathbb{G}_{m,Y} \to \mathbb{G}_{m,X}$. As $\pi'$ is a finite morphism, $\pi'_{*}$ is an exact functor $\cite[\mbox{II.3.6}]{Mil80}$. Therefore, $\pi'_{*}$ induces the map 
$Ext_{Y}^{n}(\mathcal{F},\mathbb{G}_{m,Y}) \to Ext_{X}^{n}(\pi'_{*}\mathcal{F},\pi'_{*}\mathbb{G}_{m,Y})$. We define $Nm$ to be the composition of the following maps
\[  Ext_{Y}^{n}(\mathcal{F},\mathbb{G}_{m,Y}) \xrightarrow{\pi'_{*}} Ext_{X}^{n}(\pi'_{*}\mathcal{F},\pi'_{*}\mathbb{G}_{m,Y})
\xrightarrow{N_{L/K}} Ext_{X}^{n}(\pi'_{*}\mathcal{F},\mathbb{G}_{m,X}). \]
The fact that $Nm$ is an isomorphism is the Norm theorem $\cite[\mbox{II.3.9}]{Mil06}$.
\item As $\pi'_{*}$ is an exact functor, $H^{n}_{et}(X,\pi'_{*}\mathcal{F}) \simeq H^{n}_{et}(Y,\mathcal{F})$. Therefore, $ H^{n}_{W}(X,\pi'_{*}\mathcal{F}) \simeq H^{n}_{W}(Y,\mathcal{F}) $ for $n=0,1$. In addition, by part 1, we have 
\[ H^{3}_{W}(X,\pi'_{*}\mathcal{F}) = Hom_{X}(\pi'_{*}\mathcal{F},\mathbb{G}_m)_{tor}^D 
\simeq Hom_{Y}(\mathcal{F},\mathbb{G}_m)_{tor}^D  = H^{3}_{W}(Y,\mathcal{F}).
\]
To prove $H^{2}_{W}(Y,\mathcal{F}) \simeq H^{2}_{W}(X,\pi'_{*}\mathcal{F})$, we apply the 5-lemma to the following diagram 
\[ \xymatrixrowsep{0.2in}
\xymatrix{ 0 \ar[r] & H^{2}_{et}(X,\pi'_{*}\mathcal{F})   \ar[d]^{\simeq} \ar[r]  & H^2_{W}(X,\pi'_{*}\mathcal{F}) \ar[d] \ar[r]   & Hom_{X}(\pi'_{*}\mathcal{F},\mathbb{G}_m)^{*}  \ar[d]^{\simeq} \ar[r] &  0 \\
  0 \ar[r] &  H^{2}_{et}(Y,\mathcal{F})\ar[r] &   H^2_{W}(Y,\mathcal{F}) \ar[r]  &  Hom_{Y}(\mathcal{F},\mathbb{G}_m)^{*}       \ar[r]    &  0 }\]
\item Recall from lemma $\ref{Mackey}$ that for $w|v$, we have $ (j_v)_{*}(\pi_w)_{*}=\pi'_{*}(j_w)_{*}$.
In addition, $(j_v)^{*}\pi'_{*}\mathcal{F} \simeq \prod_{w|v}(\pi_w)_{*}(j_w)^{*}\mathcal{F}$. Hence,
\begin{eqnarray*}
(\pi'_{*}\mathcal{F})_B &=& \prod_{v \in S_{K,\infty}}(j_v)_{*}(j_v)^{*}\pi'_{*}\mathcal{F} 
\quad \simeq \prod_{v \in S_{K,\infty}}(j_v)_{*}\left(\prod_{w|v}(\pi_w)_{*}(j_w)^{*}\mathcal{F}\right)  \\
&=&  \prod_{w \in S_{L,\infty}}(j_v)_{*}(\pi_w)_{*}(j_w)^{*}\mathcal{F}  \quad
= \prod_{w \in S_{L,\infty}}\pi'_{*}(j_w)_{*}(j_w)^{*}\mathcal{F} = \pi'_{*}(\mathcal{F}_B).
\end{eqnarray*}
Therefore,  
$ H^0_{et}(X,(\pi'_{*}\mathcal{F})_B) \simeq H^0_{et}(X,\pi'_{*}(\mathcal{F}_B))\simeq H^0_{et}(Y,\mathcal{F}_B)$.
\end{enumerate}
\end{proof}

\begin{prop}\label{finite_inv}
If $\mathcal{F}$ is strongly-$\mathbb{Z}$-constructible then so is $\pi'_{*}\mathcal{F}$ and $R(\pi'_{*}\mathcal{F})=R(\mathcal{F})$.
\end{prop}
\begin{proof}
As $\pi'_{*}$ preserves $\mathbb{Z}$-constructible sheaves
$\cite[\mbox{page 146}]{Mil06}$, $\pi'_{*}\mathcal{F}$ is $\mathbb{Z}$-constructible. From lemma $\ref{pre_finite_inv}$, it remains to show that the regulator pairing of $\pi'_{*}\mathcal{F}$ is non-degenerate and $R(\pi'_{*}\mathcal{F})=R(\mathcal{F})$.
They will all follow once we prove the diagram below commutes
\begin{equation}\label{finite_inv_seq}
\xymatrixrowsep{0.2in}  
  \xymatrix{ \left(\frac{H^0_{et}(X,(\pi'_{*}\mathcal{F})_{B})}{H^0_{et}(X,\pi'_{*}\mathcal{F})}\right) \ar[d]^{\psi}_{\simeq}  & \times & Hom_{X}(\pi'_{*}\mathcal{F},\mathbb{G}_m)   \ar[r]  & \mathbb{R} \ar[d]^{id}  \\
 \left(\frac{H^0_{et}(Y,\mathcal{F}_{B})}{H^0_{et}(Y,\mathcal{F})}\right) & \times & 
 Hom_{Y}(\mathcal{F},\mathbb{G}_m) \ar[u]^{Nm}_{\simeq}  \ar[r]  & \mathbb{R} }
\end{equation}
Let $\alpha$ and $\phi$ be elements of $H^0_{et}(X,(\pi'_{*}\mathcal{F})_{B})$ and $Hom_{Y}(\mathcal{F},\mathbb{G}_m)$. We need to show
\begin{equation}\label{finite_inv_eq2}
\Lambda_L \circ \phi_B(\psi(\alpha))=\Lambda_K \circ Nm(\phi)_B(\alpha).
\end{equation}
From lemma $\ref{pre_finite_inv}$ , $Nm(\phi)=N_{L/K}\circ \pi'_{*}\phi$. Let us consider the following diagram
\begin{equation}\label{finite_inv_eq3}
 \xymatrixrowsep{0.2in}\xymatrixcolsep{5pc}\xymatrix {
 &   &    &    \\
\left(\frac{H^0_{et}(X,(\pi'_{*}\mathcal{F})_{B})}{H^0_{et}(X,\pi'_{*}\mathcal{F})}\right)_{\mathbb{R}}  
 \ar[r]^-{(\pi'_{*}\phi)_B} \ar[d]^{\psi}   \ar@/^3pc/[rr]^-{Nm(\phi)_B}
 & \left(\frac{H^0_{et}(X,(\pi'_{*}\mathbb{G}_m)_{B})}{H^0_{et}(X,\pi'_{*}\mathbb{G}_m)}\right)_{\mathbb{R}}  
 \ar[r]^-{(N_{L/K})_B} \ar[d]^{\psi}
& \left(\frac{H^0_{et}(X,(\mathbb{G}_m)_{B})}{H^0_{et}(X,\mathbb{G}_m)}\right)_{\mathbb{R}}  
\ar[d]^-{\Lambda_K} \\
\left(\frac{H^0_{et}(Y,\mathcal{F}_{B})}{H^0_{et}(Y,\mathcal{F})}\right)_{\mathbb{R}}
\ar[r]^-{\phi_B}
& \left(\frac{H^0_{et}(Y,(\mathbb{G}_m)_{B})}{H^0_{et}(Y,\mathbb{G}_m)}\right)_{\mathbb{R}} 
\ar[r]^-{\Lambda_L}  \ar[ur]^{N_{L/K}}
& \mathbb{R} 
}
\end{equation}
The left square of ($\ref{finite_inv_eq3}$) commutes by functoriality. It is not hard to see the upper triangle on the right is commutative. We shall prove that the lower triangle on the right also commutes. Let $\beta$ be an element of $H^0_{et}(Y,(\mathbb{G}_m)_{B}) \simeq \prod_{w \in S_{L,\infty}} L_w^{*}$. Then 
\begin{eqnarray*}
\Lambda_K(N_{L/K}(\beta)) &=& \sum_{v \in S_{K,\infty}} \log|N_{L/K}(\beta)_v|_v 
= \sum_{v \in S_{K,\infty}} \sum_{w|v} \log|N_{L_w/K_v}(\beta_w)|_v  \\
&=& \sum_{v \in S_{K,\infty}} \sum_{w|v} \log|\beta_w|_w   
= \sum_{w \in S_{L,\infty}}  \log|\beta_w|_w   
=  \Lambda_L(\beta).
\end{eqnarray*}
Therefore, diagram ($\ref{finite_inv_eq3}$) is commutative and from this we deduce equation ($\ref{finite_inv_eq2}$). 
As a result, diagram ($\ref{finite_inv_seq}$) commutes. Hence, the proposition is proved.
\end{proof}

\begin{cor}\label{cor_finite_inv}
$\pi'_{*}\mathbb{Z}$ is a strongly-$\mathbb{Z}$-constructible sheaf.
\end{cor}
We need the following theorem of Ono.
\begin{thm}\label{thm_ono}
Let $M$ be a discrete $G_K$-module. Then there exist finitely many Galois extensions $\{K_{\mu}\}_{\mu}$, $\{K_{\lambda}\}_{\lambda}$ of $K$ and positive integers $n$,
$\{m_{\mu}\}_{\mu}$, $\{m_{\lambda}\}_{\lambda}$ and a finite $G_K$-module N such that 
\[ 0 \to M^{n}\oplus \prod_{\mu}(\pi_{\mu})_{*}\mathbb{Z}^{m_{\mu}} \to \prod_{\lambda}(\pi_{\lambda})_{*}\mathbb{Z}^{m_{\lambda}} \to N \to 0 \]
is an exact sequence of $G_K$-modules. Here $\pi_{\mu}$ is the natural map from $Spec(K_{\mu})$ to $Spec(K)$.
\end{thm}

\begin{proof}
For a proof, see $\cite[\mbox{1.5.1}]{Ono61}$. 
\end{proof}
The following proposition is a sheaf-theoretic version of theorem $\ref{thm_ono}$.
\begin{prop}\label{ono_etale}
Let $M$ be a discrete $G_K$-module which is also a finitely generated abelian group. Then there exist finitely many Galois extensions $\{K_{\mu}\}_{\mu}$, $\{K_{\lambda}\}_{\lambda}$ of $K$ and positive integers $n$,
$\{m_{\mu}\}_{\mu}$, $\{m_{\lambda}\}_{\lambda}$, a constructible sheaf $\mathcal{R}$ and a finite $G_K$-module $N$ such that we have the following exact sequences 
\begin{equation}\label{ono_etale_seq1}
0 \to M^{n}\oplus \prod_{\mu}(\pi_{\mu})_{*}\mathbb{Z}^{m_{\mu}} \to \prod_{\lambda}(\pi_{\lambda})_{*}\mathbb{Z}^{m_{\lambda}} \to N \to 0,
\end{equation}
\begin{equation}\label{ono_etale_seq2}
0 \to (j_{*}M)^{n}\oplus \prod_{\mu}(\pi_{\mu}')_{*}\mathbb{Z}^{m_{\mu}} \to \prod_{\lambda}(\pi_{\lambda}')_{*}\mathbb{Z}^{m_{\lambda}} \to \mathcal{R} \to 0
\end{equation}
where $\pi_{\mu}':Spec(O_{K_{\mu}}) \to Spec(O_K)$ and $\pi_{\mu}: Spec(K_{\mu})\to 
Spec(K)$ are the natural maps.
\end{prop}
\begin{proof}
The existence of ($\ref{ono_etale_seq1}$) is precisely theorem $\ref{thm_ono}$. 
Let $P_1=\prod_{\mu}(\pi_{\mu})_{*}\mathbb{Z}^{m_{\mu}}$ and $P_2=\prod_{\lambda}(\pi_{\lambda})_{*}\mathbb{Z}^{m_{\lambda}}$. 
By applying $j_{*}$ to ($\ref{ono_etale_seq1}$), we obtain the exact sequence 
\[ 0 \to j_{*}M^n\oplus j_{*}P_1 \to j_{*}P_2 \to j_{*}N \to R^{1}j_{*}(M\oplus P_1)\]
which we split into two exact sequences
\[ 0 \to j_{*}M^n\oplus j_{*}P_1 \to j_{*}P_2 \to \mathcal{R} \to 0 \qquad \mbox{and} \qquad
 0 \to \mathcal{R} \to j_{*}N \to \mathcal{Q} \to 0 \]
where $\mathcal{Q}$ is a subsheaf of $R^{1}j_{*}(M\oplus P_1)$. As $R^{1}j_{*}(M\oplus P_1)$ is negligible, $\mathcal{Q}$ is constructible. Since $j_{*}P_1=\prod_{\mu}(\pi_{\mu}')_{*}\mathbb{Z}^{m_{\mu}}$ and $j_{*}P_2=\prod_{\lambda}(\pi_{\lambda}')_{*}\mathbb{Z}^{m_{\lambda}}$, we have $j_{*}P_1$ and $j_{*}P_2$ are strongly-$\mathbb{Z}$-constructible by corollary $\ref{cor_finite_inv}$. As $N$ is finite, $j_{*}N$ is constructible. Since $\mathcal{Q}$ is constructible, $\mathcal{R}$ is constructible as the category of constructible sheaves is abelian. 
\end{proof}
\begin{prop}\label{jM_strong_constr}
Let $M$ be a discrete $G_K$-module. Then $j_{*}M$ is a strongly-$\mathbb{Z}$-constructible sheaf.
\end{prop}
\begin{proof}
Consider sequence ($\ref{ono_etale_seq2}$) of proposition $\ref{ono_etale}$.
As $\mathcal{R}$ is constructible and $\prod_{\lambda}(\pi_{\lambda}')_{*}\mathbb{Z}^{m_{\lambda}}$ is strongly-$\mathbb{Z}$-constructible, $(j_{*}M)^{n}\oplus \prod_{\mu}(\pi_{\mu}')_{*}\mathbb{Z}^{m_{\mu}}$ is strongly-$\mathbb{Z}$-constructible by proposition $\ref{abel_constr}$. Since $\prod_{\mu}(\pi_{\mu}')_{*}\mathbb{Z}^{m_{\mu}}$ is strongly-$\mathbb{Z}$-constructible, so is  $j_{*}M$.
\end{proof}

\section{Euler Characteristics Of Strongly-$\mathbb{Z}$-Constructible Sheaves}
\label{chapter_euler_strongly_Z}

\subsection{Construction} 
Let $\mathcal{F}$ be a strongly-$\mathbb{Z}$-constructible sheaf on $X$. 
There are natural maps
$ R\Gamma_{W}(X,\mathcal{F}) \to \tau_{\leq 1} R\Gamma_{et}(X,\mathcal{F}) $ and 
$ \tau_{\leq 1} R\Gamma_{et}(X,\mathcal{F}) \to R\Gamma_{et}(X,\mathcal{F})$.
The morphism $\mathcal{F}\to \mathcal{F}_{B}$ induces a map of complexes
$ R\Gamma_{et}(X,\mathcal{F}) \to R\Gamma_{et}(X,\mathcal{F}_{B}) $.
Composing all three maps yields the map
\begin{equation}\label{defn_cF}
R\Gamma_W(X,\mathcal{F}) \xrightarrow{} R\Gamma_{et}(X,\mathcal{F}_B).
\end{equation}
\begin{defn}\label{defn_dF}
We define the complex $D_{\mathcal{F}}$ by the cone
\[ D_{\mathcal{F}} := [ R\Gamma_W(X,\mathcal{F}) \xrightarrow{} R\Gamma_{et}(X,\mathcal{F}_B) ][-1]. \]
\end{defn}
 
 \begin{prop}\label{Hn_dF}
Let $\mathcal{F}$ be a strongly-$\mathbb{Z}$-constructible sheaf on $X$. Then $H^n(D_\mathcal{F})$ satisfy 
\begin{equation}\label{Hn_dF_seq1}
0 \to H^1(D_\mathcal{F}) \to H^0_{et}(X,\mathcal{F}) \to H^0_{et}(X,\mathcal{F}_B) \to H^2(D_\mathcal{F}) \xrightarrow{\beta} H^1_{et}(X,\mathcal{F}) \to 0
\end{equation}
\begin{equation}\label{Hn_dF_seq2}
0 \to H^2_{et}(X,\mathcal{F}) \to H^3(D_\mathcal{F}) \to Hom_X(\mathcal{F},\mathbb{G}_m)^{*} \to 0
\end{equation}
and $H^4(D_\mathcal{F})=Hom_{X}(\mathcal{F},\mathbb{G}_m)_{tor}^D$ and $H^n(D_{\mathcal{F}})=0$ otherwise. 
 \end{prop}
 \begin{proof}
There is a distinguished triangle
 \[ D_\mathcal{F}[1] \to R\Gamma_W(X,\mathcal{F}) \to R\Gamma_{et}(X,\mathcal{F}_B) \to D_\mathcal{F}[2]. \]
 The long exact sequence of cohomology of this triangle yields the following exact sequence  
 \[ 0 \to H^1(D_\mathcal{F}) \to  H^0_{et}(X,\mathcal{F}) \to H^0_{et}(X,\mathcal{F}_B) \to H^2(D_\mathcal{F})
 \to H^1_{et}(X,\mathcal{F}) \to 0 \]
 and $H^{n+1}(D_{\mathcal{F}}) \simeq H^n_{W}(X,\mathcal{F})$ for $n \geq 2$. The lemma then follows from proposition $\ref{Weil_F}$.
 \end{proof}
 
\begin{lemma}\label{M_det}
Let $\beta$ be the map $H^2(D_\mathcal{F}) \to  H^1_{et}(X,\mathcal{F})$ from proposition $\ref{Hn_dF}$. Then 
there is a canonical isomorphism $\theta : H^2(D_\mathcal{F})_{\mathbb{R}} \to H^3(D_\mathcal{F})_{\mathbb{R}}$ with 
$|\det\theta|=R(\mathcal{F})/[\mathrm{cok}(\beta_{tor})]$ with respect to integral bases.
\end{lemma}

\begin{proof}
We construct the isomorphism $\theta : H^2(D_\mathcal{F})_{\mathbb{R}} \to H^3(D_\mathcal{F})_{\mathbb{R}}$ as follows. From the exact sequence ($\ref{Hn_dF_seq1}$)
 we have \[ 0 \to H^0_{et}(X,\mathcal{F}_B)/H^0_{et}(X,\mathcal{F}) \to H^2(D_\mathcal{F}) \xrightarrow{\beta} H^1_{et}(X,\mathcal{F}) \to 0. \]
As $H^1_{et}(X,\mathcal{F})$ is finite, there is an isomorphism
$ \phi : (H^0_{et}(X,\mathcal{F}_B)/H^0_{et}(X,\mathcal{F}))_{\mathbb{R}} \to H^2(D_\mathcal{F})_{\mathbb{R}} $.
By lemma $\ref{det_tor_3term}$, 
$ |\det(\phi)|=[\mathrm{cok}(\beta_{tor})]$ with respect to integral bases.

From the exact sequence ($\ref{Hn_dF_seq2}$)
and the fact that $H^2_{et}(X,\mathcal{F})$ is finite, there is an isomorphism 
$ \psi : H^3(D_\mathcal{F})_{\mathbb{R}} \to Hom_X(\mathcal{F},\mathbb{G}_m)^{*}_{\mathbb{R}} $.
By lemma $\ref{det_tor_3term}$ and the fact that $Hom_X(\mathcal{F},\mathbb{G}_m)^{*}$ is torsion free, we have 
$|\det(\psi)|=1$ with respect to integral bases.

As $\mathcal{F}$ is strongly-$\mathbb{Z}$-constructible, the regulator pairing induces the isomorphism  
\[ \gamma : \left(\frac{H^0_{et}(X,\mathcal{F}_B)}{H^0_{et}(X,\mathcal{F})}\right)_{\mathbb{R}} \to Hom_X(\mathcal{F},\mathbb{G}_m)^{*}_{\mathbb{R}}.\]
By definition $\ref{defn_regulator}$, $|\det(\gamma)|=R(\mathcal{F})$ with respect to integral bases.

We define $\theta$ to be $\psi^{-1} \circ \beta \circ \phi^{-1}$. 
Therefore, $\theta : H^2(D_\mathcal{F})_{\mathbb{R}} \to H^3(D_\mathcal{F})_{\mathbb{R}}$ is an isomorphism 
and $|\det\theta|=R(\mathcal{F})/[\mathrm{cok}(\beta_{tor})]$ with respect to integral bases.
\end{proof}
The existence of $\theta$ in lemma $\ref{M_det}$ enables us to make the following definition.
\begin{defn}\label{euler_defn1}
For a strongly-$\mathbb{Z}$-constructible sheaf $\mathcal{F}$ on $X$, we define the Euler characteristic $\chi(\mathcal{F})$ by 
\[ {\chi}(\mathcal{F}):= \frac{[H^1(D_\mathcal{F})][H^3(D_\mathcal{F})_{tor}]}{[H^2(D_\mathcal{F})_{tor}][H^4(D_\mathcal{F})]}|\det(\theta)| \]
where $\theta$ is the isomorphism constructed in the lemma $\ref{M_det}$ and its determinant is computed with respect to integral bases.
\end{defn}

\begin{prop}\label{euler_comp1}
For a strongly-$\mathbb{Z}$-constructible sheaf $\mathcal{F}$ on $X$, we have 
\[ {\chi}(\mathcal{F})= \frac{[H^0_{et}(X,\mathcal{F})_{tor}][H^2_{et}(X,\mathcal{F})]R(\mathcal{F})}{[H^1_{et}(X,\mathcal{F})][Hom_X(\mathcal{F},\mathbb{G}_m)^D_{tor}][H^0_{et}(X,\mathcal{F}_B)_{tor}][\mathrm{cok}(\delta_{tor})]} \]
where $\delta$ is the quotient map $H^0_{et}(X,\mathcal{F}_B)\to H^0_{et}(X,\mathcal{F}_B)/H^0_{et}(X,\mathcal{F})$. In particular, if $\mathcal{F}$ is a constructible sheaf then
 \[ {\chi}(\mathcal{F})= \frac{[H^0_{et}(X,\mathcal{F})][H^2_{et}(X,\mathcal{F})]}
 {[H^1_{et}(X,\mathcal{F})][H^3_{et}(X,\mathcal{F})][H^0_{et}(X,\mathcal{F}_B)]}. \]
\end{prop}

\begin{proof}
By lemma $\ref{M_det}$, $|\det\theta|=R(\mathcal{F})/[\mathrm{cok}(\beta_{tor})]$ where 
$\beta : H^2(D_{\mathcal{F}}) \to H^1_{et}(X,\mathcal{F}) $.
Now we compute the torsion subgroups of $H^n(D_\mathcal{F})$. 
We have $[H^3(D_\mathcal{F})_{tor}]=[H^2_{et}(X,\mathcal{F})]$ and $[H^4(D_\mathcal{F})]=[Hom_X(\mathcal{F},\mathbb{G}_m)^D_{tor}]$ from proposition $\ref{Hn_dF}$.
We split the exact sequence $(\ref{Hn_dF_seq1})$ into
 \[ 0 \to H^1(D_\mathcal{F}) \to H^0_{et}(X,\mathcal{F}) \to H^0_{et}(X,\mathcal{F}_B) \xrightarrow{\delta} 
 \frac{H^0_{et}(X,\mathcal{F}_B)}{H^0_{et}(X,\mathcal{F})} \to 0, \]
 \[ 0 \to \frac{H^0_{et}(X,\mathcal{F}_B)}{H^0_{et}(X,\mathcal{F})} \to H^2(D_\mathcal{F}) \xrightarrow{\beta} H^1_{et}(X,\mathcal{F}) \to 0 .\]
Applying lemma $\ref{torsion_group}$ to the two exact sequences above, we have  
 \[ [H^1(D_\mathcal{F})] = \frac{[{H^0_{et}(X,\mathcal{F})}_{tor}]
 \left[\left(\frac{H^0_{et}(X,\mathcal{F}_B)}{H^0_{et}(X,\mathcal{F})}\right)_{tor}\right]}
 {[{H^0_{et}(X,\mathcal{F}_B)}_{tor}][\mathrm{cok}(\delta_{tor})]},  \qquad
  [H^2(D_\mathcal{F})_{tor}] = \frac{\left[\left(\frac{H^0_{et}(X,\mathcal{F}_B)}{H^0_{et}(X,\mathcal{F})}\right)_{tor}\right][H^1_{et}(X,\mathcal{F})]}{[\mathrm{cok}(\beta_{tor})]} .\]
Putting everything together establishes the formula for $\chi(\mathcal{F})$.
 \end{proof}
 
\subsection{Simple Computations}

\begin{prop}\label{Z_Euler}
Let $h$, $R$ and $w$ be the class number, the regulator and the number of roots of unity of $K$ respectively. Then
$\chi(\mathbb{Z})={hR}/{w} $.
In particular, $\zeta_K^{*}(0)=-\chi(\mathbb{Z})$.
\end{prop}
\begin{proof}
From theorem $\ref{et_ZGm}$, $H^1_{et}(X,\mathbb{Z})=0$ and $[H^2_{et}(X,\mathbb{Z})]=h$. It is clear that $H^0_{et}(X,\mathbb{Z}_B)_{tor}=0$ and $\mathrm{cok}(\delta_{tor})=0$.
In addition, $R(\mathbb{Z})=R$ and $[(O_K^{*})_{tor}]=w$. As a result,
${\chi}(\mathbb{Z})={hR}/{w}$.
Therefore, $\zeta_K^{*}(0) = -\chi(\mathbb{Z}) $ by the analytic class number formula.
\end{proof}
\begin{prop}\label{Zn_Euler}
Euler characteristics of finite constant sheaves are 1.
\end{prop}
\begin{proof}
 It suffices to prove this proposition for the constant sheaf $\mathbb{Z}/n$.
From theorem $\ref{et_ZGm}$, the \'etale cohomology of $\mathbb{Z}/n$ is given by
\[  H^r_{et}(X,\mathbb{Z}/n) = 
         \left\{
				\begin{array}{ll}
 					\mathbb{Z}/n   &  \mbox{$r=0$}\\
 					(Pic(O_K)/n)^D &  \mbox{$r=1$ }\\
 						\mu_n(K)^D & \mbox{$r=3$}. 
				\end{array}
			   \right.
	            \]
\[ 0 \to Pic(O_K)[n]^D \to H^2_{et}(X,\mathbb{Z}/n) \to (O_K^{*}/(O_K^{*})^{n})^D \to 0 .\]
Observe that $[H^0_{et}(X,(\mathbb{Z}/n)_B)]=n^{|S_{\infty}|}$ and $R(\mathbb{Z}/n)=1$. From proposition $\ref{euler_comp1}$ and Dirichlet's Unit Theorem,
${\chi}(\mathbb{Z}/n)=1$.
\end{proof}
\begin{prop}\label{Euler_negligible}
Euler characteristics of negligible sheaves are 1.
\end{prop}
\begin{proof} 
It is enough to prove this lemma for the sheaf $i_{*}M$ where $M$ is a finite $\hat{\mathbb{Z}}$-module. 
We have $H^{n}_{et}(X,i_{*}M) \simeq H^{n}(\hat{\mathbb{Z}},M)$ which is 0 for $n \geq 2$ 
$\cite[\mbox{page 189}]{Ser95}$. Moreover, $[H^0(\hat{\mathbb{Z}},M)]=[H^1(\hat{\mathbb{Z}},M)]$ for finite $M$. Also $(i_{*}M)_B=0$, therefore ${\chi}(i_{*}M)=1$.
\end{proof}

\begin{prop}\label{euler_finite_inv}
Let $L$ be a finite Galois extension of $K$, $Y=Spec(O_L)$ and $\pi' : Y \to X $ be the natural map. Let $\mathcal{F}$ be a strongly-$\mathbb{Z}$-constructible sheaf on $Y$. Then $\pi'_{*}F$ is a strongly-$\mathbb{Z}$-constructible sheaf on $X$. Moreover, $H^n(D_{\pi'_{*}\mathcal{F}}) \simeq H^n(D_{\mathcal{F}})$ and ${\chi}(\pi'_{*}F)={\chi}(F)$.
\end{prop}
\begin{proof}
From proposition $\ref{finite_inv}$, $\pi'_{*}F$ is a strongly-$\mathbb{Z}$-constructible sheaf and $R(\pi'_{*}\mathcal{F})=R(\mathcal{F})$. To show $H^n(D_{\pi'_{*}\mathcal{F}}) \simeq H^n(D_{\mathcal{F}})$, let us consider the following commutative diagram
\[ \xymatrix @C=1.1pc{0 \ar[r]& H^1(D_{\pi'_{*}\mathcal{F}}) \ar[r] \ar[d] & H^0_{et}(X,\pi'_{*}\mathcal{F}) \ar[r] \ar[d]^{\simeq}& H^0_{et}(X,(\pi'_{*}\mathcal{F})_B) \ar[r] \ar[d]^{\simeq} & H^2(D_{\pi'_{*}\mathcal{F}}) \ar[r]\ar[d] & H^1_{et}(X,\pi'_{*}\mathcal{F}) \ar[r]\ar[d]^{\simeq} & 0 \\ 
0 \ar[r]& H^1(D_\mathcal{F}) \ar[r]  & H^0_{et}(Y,\mathcal{F}) \ar[r] & H^0_{et}(Y,\mathcal{F}_B) \ar[r] & H^2(D_\mathcal{F}) \ar[r] & H^1_{et}(Y,\mathcal{F}) \ar[r]  & 0 \\ }
\]
where the rows are exact from proposition $\ref{Hn_dF}$. Note that the map in the center is an isomorphism by proposition $\ref{finite_inv}$. Thus, the 5-lemma implies that  
$H^n(D_{\pi'_{*}\mathcal{F}}) \simeq H^n(D_{\mathcal{F}})$ for $n=1,2$. For $n=3,4$, again from proposition $\ref{Hn_dF}$, we have $H^n(D_{\pi'_{*}\mathcal{F}}) \simeq H^n_{W}(X,\pi'_{*}\mathcal{F})$ and $H^n(D_{\mathcal{F}})\simeq H^n_{W}(Y,\mathcal{F})$. Thus, $H^n(D_{\pi'_{*}\mathcal{F}}) \simeq H^n(D_{\mathcal{F}})$ for $n=3,4$ by lemma $\ref{pre_finite_inv}$. 
Therefore, ${\chi}(\pi'_{*}F)={\chi}(F)$.

\end{proof}

\begin{cor}\label{euler_piZ}
The sheaf $\pi_{*}\mathbb{Z}$ on $Spec(K)$  corresponds to the induced $G_K$-module ${Ind}_{G_L}^{G_K}(\mathbb{Z})$. If we write $\pi_{*}\mathbb{Z}$ for ${Ind}_{G_L}^{G_K}(\mathbb{Z})$ then
 $ L^{*}(K,\pi_{*}\mathbb{Z},0)=\pm \chi(\pi'_{*}\mathbb{Z}) $.
\end{cor}
\begin{proof}
By proposition $\ref{euler_finite_inv}$, $\chi(\pi'_{*}\mathbb{Z})=\chi(\mathbb{Z})$. Also by proposition 
$\ref{Z_Euler}$, 
$\zeta^{*}_L(0)=\pm \chi(\mathbb{Z})$. Since $ L(K,\pi_{*}\mathbb{Z},s)=\zeta_L(s)$, $L^{*}(K,\pi_{*}\mathbb{Z},0)=\pm \chi(\pi'_{*}\mathbb{Z})$.
\end{proof}
\subsection{Multiplicative Property}
\begin{lemma}\label{cok_delta_tor}
Let $\delta$ be the map $H_{et}^0(X,\mathcal{F}_B) \to H_{et}^0(X,\mathcal{F}_B)/H^0_{et}(X,\mathcal{F})$. Then the exact sequence
\begin{equation}\label{R(F)_T}
0 \to   H^0_{et}(X,\mathcal{F})_{\mathbb{R}} \xrightarrow{\Delta}  
H_{et}^0(X,\mathcal{F}_B)_{\mathbb{R}} 
\xrightarrow{\delta} \left(\frac{H^0_{et}(X,\mathcal{F}_B)}{H^0_{et}(X,\mathcal{F})}\right)_{\mathbb{R}}\to 0
\end{equation}
has determinant $[\mathrm{cok}(\delta_{tor})]$ with respect to integral bases.
\end{lemma}

\begin{proof}
Since $\mathcal{F}$ is strongly-$\mathbb{Z}$-constructible, the kernel $H^1(D_{\mathcal{F}})$ of the map $H^0_{et}(X,\mathcal{F}) \to H^0_{et}(X,\mathcal{F}_B)$ is finite. Applying lemma $\ref{det_tor_5term}$ to the exact sequence
\[ 0 \to H^1(D_{\mathcal{F}}) \to H^0_{et}(X,\mathcal{F}) \xrightarrow{\Delta} H^0_{et}(X,\mathcal{F}_B) \xrightarrow{\delta} 
\frac{H^0_{et}(X,\mathcal{F}_B)}{H^0_{et}(X,\mathcal{F})} \to 0 \] 
we deduce that ($\ref{R(F)_T}$) has determinant $[\mathrm{cok} (\delta_{tor})]$ 
with respect to integral bases.
\end{proof}
The main result of this section is the following theorem. 
\begin{thm}\label{mul}
Suppose we have a short exact sequence of strongly-$\mathbb{Z}$-constructible sheaves
\begin{equation}\label{chi_mul_seq1}
0 \to \mathcal{F}_1 \to \mathcal{F}_2 \to \mathcal{F}_3 \to 0.
\end{equation}
   Then $ {\chi}(\mathcal{F}_2) = {\chi}(\mathcal{F}_1){\chi}(\mathcal{F}_3).$
\end{thm}

\begin{proof}
From proposition $\ref{euler_comp1}$, ${{\chi}(\mathcal{F}_1){\chi}(\mathcal{F}_3)}/{{\chi}(\mathcal{F}_2)}$ is given by

\begin{equation}\label{chi_mul_eq0}
\frac{\left(\prod_{i=1}^3[H^0_{et}(X,\mathcal{F}_i)_{tor}]^{(-1)^{i+1}}\right)
\left(\prod_{i=1}^3R(\mathcal{F}_i)^{(-1)^{i+1}}\right)
\left(\prod_{i=1}^3\left(\frac{[H^2_{et}(X,\mathcal{F}_{i+1})]}{[H^1_{et}(X,\mathcal{F}_i)]}\right)^{(-1)^{i+1}}\right)}
{ \left(\prod_{i=1}^3[H^0_{et}(X,\mathcal{F}_{i,B})_{tor}]^{(-1)^{i+1}}\right)
\left(\prod_{i=1}^3[Hom_{X}(\mathcal{F}_i,\mathbb{G}_m)_{tor}]^{(-1)^{i+1}}\right)
\left(\prod_{i=1}^3[\mathrm{cok}(\delta_{i,tor})]^{(-1)^{i+1}}\right)} 
\end{equation}

We split the long exact sequence of cohomology associated with ($\ref{chi_mul_seq1}$) into the following exact sequences 
\begin{equation}\label{chi_mul_seq2}
(\mathcal{H}^0): \qquad 0 \to H^0_{et}(X,\mathcal{F}_1) \to H^0_{et}(X,\mathcal{F}_2) \to
 H^0_{et}(X,\mathcal{F}_3) \to S \to 0 ,
\end{equation}
\begin{equation}\label{chi_mul_seq3}
0 \to S \to H^1_{et}(X,\mathcal{F}_1) \to H^1_{et}(X,\mathcal{F}_2) \to ... \to H^2_{et}(X,\mathcal{F}_3) \to Q \to 0 ,
\end{equation}
where both $S$ and $Q$ are finite abelian groups. From ($\ref{chi_mul_seq3}$), we have 
\begin{equation}\label{chi_mul_eq1}
\prod_{i=1}^3\left(\frac{[H^2_{et}(X,\mathcal{F}_i)]}{[H^1_{et}(X,\mathcal{F}_i)]}\right)^{(-1)^{i+1}}
=\frac{[Q]}{[S]} .
\end{equation}
Applying lemma $\ref{det_tor}$ to ($\ref{chi_mul_seq2}$), we obtain
\begin{equation}\label{chi_mul_eq2}
\left(\prod_{i=1}^3[H^0_{et}(X,\mathcal{F}_i)_{tor}]^{(-1)^{i+1}}\right)=
\nu(\mathcal{H}^0)_{\mathbb{R}}[S].
\end{equation}
As all the $\mathcal{F}_i$ are strongly-$\mathbb{Z}$-constructible, $ Ext^1_X(\mathcal{F}_i,\mathbb{G}_m)^D \simeq H^2_{et}(X,\mathcal{F}_i)$. Therefore, the following sequence is exact
\begin{equation}\label{chi_mul_seq4}
(\mathcal{H}om): \quad 0 \to Hom_X(\mathcal{F}_3,\mathbb{G}_m) \to Hom_X(\mathcal{F}_2,\mathbb{G}_m) \to Hom_X(\mathcal{F}_1,\mathbb{G}_m) \to Q^D \to 0.
\end{equation}
Applying lemma $\ref{det_tor}$ to ($\ref{chi_mul_seq4}$) yields
\begin{equation}\label{chi_mul_eq3}
\left(\prod_{i=1}^3[Hom_{X}(\mathcal{F}_i,\mathbb{G}_m)_{tor}]^{(-1)^{i+1}}\right)
=\nu(\mathcal{H}om)_{\mathbb{R}}[Q^D].
\end{equation}
Let $(\mathcal{H}_B)$ be the exact sequence
\begin{equation}\label{chi_mul_seq5}
(\mathcal{H}_B): \quad 0 \to H^0_{et}(X,\mathcal{F}_{1,B}) \to H^0_{et}(X,\mathcal{F}_{2,B}) \to H^0_{et}(X,\mathcal{F}_{3,B}) \to 0.
\end{equation} 
Applying lemma $\ref{det_tor}$ to ($\ref{chi_mul_seq5}$), we obtain
\begin{equation}\label{chi_mul_eq4}
 \left(\prod_{i=1}^3[H^0_{et}(X,\mathcal{F}_{i,B})_{tor}]^{(-1)^{i+1}}\right)=\nu(\mathcal{H}_B)_{\mathbb{R}}.
\end{equation}
Applying lemma $\ref{det_2x3}$ to the following diagram
\[ \xymatrix{ 
0 \ar[r] & \left(\frac{H^0_{et}(X,\mathcal{F}_{1,B})}{H^0_{et}(X,\mathcal{F}_1)}\right)_{\mathbb{R}} 
 \ar[r] \ar[d] & 
\left(\frac{H^0_{et}(X,\mathcal{F}_{2,B})}{H^0_{et}(X,\mathcal{F}_2)}\right)_{\mathbb{R}}  \ar[r] \ar[d] & 
\left(\frac{H^0_{et}(X,\mathcal{F}_{3,B})}{H^0_{et}(X,\mathcal{F}_3)}\right)_{\mathbb{R}} \ar[d] \ar[r] & 0
& (\mathcal{H}_B/\mathcal{H}^0)  \\
0 \ar[r] & Hom_{X}(\mathcal{F}_1,\mathbb{G}_m)^{*}_{\mathbb{R}} \ar[r] & Hom_{X}(\mathcal{F}_2,\mathbb{G}_m)^{*}_{\mathbb{R}} \ar[r] & Hom_{X}(\mathcal{F}_3,\mathbb{G}_m)^{*}_{\mathbb{R}} \ar[r] & 0 
& (\mathcal{H}om)^{*}_{\mathbb{R}} \\}
\]
yields 
\begin{equation}
\left(\prod_{i=1}^3R(\mathcal{F}_i)^{(-1)^{i+1}}\right)
={\nu(\mathcal{H}_B/\mathcal{H}^0)\nu(\mathcal{H}om)_{\mathbb{R}}}  .
\end{equation}
 Applying lemma $\ref{det_3x3}$ to the diagram below where all the columns are short exact sequences
\[\xymatrixrowsep{0.2in}\xymatrix{      
0 \ar[r] & H^0_{et}(X,\mathcal{F}_1)_{\mathbb{R}}  \ar[d] \ar[r]   &  H^0_{et}(X,\mathcal{F}_2)_{\mathbb{R}} \ar[d] 
\ar[r] &H^0_{et}(X,\mathcal{F}_3)_{\mathbb{R}}       \ar[d] \ar[r] & 0 & (\mathcal{H}^0)_{\mathbb{R}} 
\\
0 \ar[r] & H^0_{et}(X,\mathcal{F}_{1,B})_{\mathbb{R}}  \ar[d]^{(\delta_1)_{\mathbb{R}}} \ar[r] & H^0_{et}(X,\mathcal{F}_{2,B})_{\mathbb{R}} \ar[d]^{(\delta_2)_{\mathbb{R}}} \ar[r] & H^0_{et}(X,\mathcal{F}_{3,B})_{\mathbb{R}} \ar[d]^{(\delta_3)_{\mathbb{R}}}\ar[r] & 0 & (\mathcal{H}_B)_{\mathbb{R}}
\\
0 \ar[r] & \left(\frac{H^0_{et}(X,\mathcal{F}_{1,B})}{H^0_{et}(X,\mathcal{F}_1)}\right)_{\mathbb{R}}  \ar[r] & 
\left(\frac{H^0_{et}(X,\mathcal{F}_{2,B})}{H^0_{et}(X,\mathcal{F}_2)}\right)_{\mathbb{R}}  \ar[r]  & 
\left(\frac{H^0_{et}(X,\mathcal{F}_{3,B})}{H^0_{et}(X,\mathcal{F}_3)}\right)_{\mathbb{R}} \ar[r] & 0 &
 (\mathcal{H}_B/\mathcal{H}^0)  \\                     
  & (\mathcal{E}_1) & (\mathcal{E}_2) & (\mathcal{E}_3)   &                }
 \]
 and use lemma $\ref{cok_delta_tor}$, we have  
 \begin{equation}
 \left(\prod_{i=1}^3[\mathrm{cok}(\delta_{i,tor})]^{(-1)^{i+1}}\right) =
 \frac{\nu(\mathcal{H}^0)_{\mathbb{R}}\nu(\mathcal{H}_B/\mathcal{H}^0)}{\nu(\mathcal{H}_B)_{\mathbb{R}}}.
 \end{equation}
 Putting everything together, we obtain $\chi(\mathcal{F}_2)=\chi(\mathcal{F}_1)\chi(\mathcal{F}_3)$.
\end{proof}
 
\begin{prop} \label{euler_constructible}
Euler characteristics of constructible sheaves are 1.
\end{prop}
\begin{proof}
Let $\mathcal{F}$ be a constructible sheaf on $X$. Then there exists an open dense subset $U$ of $X$ such that $\mathcal{F}_U:=\rho^{*}\mathcal{F}$ is locally constant where $\rho :U\to X$ is the inclusion map. Let $\pi : V \to U$ be the finite \'etale morphism such that $\pi^{*}\mathcal{F}_{U}$ is a constant sheaf. Let $S$ be the set of primes of $K$ (including the infinite primes) not corresponding to a point of $U$. Let $K_S$ be the maximal subfield of $\bar{K}$ that is ramified over $K$ at only primes in $S$. Let $G_S$ be $Gal(K_S/K)$. The category of locally constant sheaves with finite stalks on $U$ is equivalent to the category of discrete finite $G_S$-modules $\cite[\mbox{page 156}]{Mil80}$. Let $M$ be the $G_S$-module corresponding to $\mathcal{F}_{U}$. In particular, $M$ is a finite abelian group. By making $U$ smaller if necessary, we may assume $[M]$ is not divisible by the residue characteristics of any closed points of $U$. From 
$\cite[\mbox{II.2.9 }]{Mil06}$, $H^n(U,\mathcal{F}_U)\simeq H^n(G_S,M)$. 
Let $i_p :p \to X$ and $\mathcal{F}_p=i_p^{*}\mathcal{F}$. We have the canonical exact sequence
\[ 0 \to \rho_{!}\mathcal{F}_U \to \mathcal{F} \to \prod_{p \in X-U}(i_p)_{*}\mathcal{F}_p \to 0. \] 
Since $\prod_{p \in X-U}(i_p)_{*}\mathcal{F}_p$ is negligible, its Euler characteristic is 1. 
By theorem $\ref{mul}$ , it suffices to show 
${\chi}(\rho_{!}\mathcal{F}_U)=1$. Since $(\rho_{!}\mathcal{F}_U)_B = \mathcal{F}_B$, 
$ [H^0_{et}(X,(\rho_{!}\mathcal{F}_U)_B)] = \prod_{v \in S_{\infty}} [M]$. Let $H^n_c(U,\mathcal{F}_U)$ be the cohomology with compact support defined in $\cite[\mbox{page 165}]{Mil06}$. As $K$ is totally imaginary, $H^n_c(U,\mathcal{F})=H^n(X,\rho_{!}\mathcal{F}_U)$. Then the proof is complete because
by $\cite[\mbox{II.2.13}]{Mil06}$, 
 \[ {\chi}(\rho_{!}\mathcal{F}_U)
 = \frac{[H^0_{c}(U,\mathcal{F}_U)][H^2_{c}(U,\mathcal{F}_U)]}
{[H^1_{c}(U,\mathcal{F}_U)][H^3_{c}(U,\mathcal{F}_U)]\prod_{v \in S_{\infty}}[M]} =1.
\]
\end{proof}
\subsection{Special Values Of L-Functions At Zero}
\begin{thm}\label{euler_value_L0}
Let $K$ be a totally imaginary number field. Let $M$ be a discrete $G_K$-module. Then 
\begin{enumerate}
	\item $\mathrm{ord}_{s=0}L(M,s)= 
	\mathrm{rank}_{\mathbb{Z}} Hom_{X}(j_{*}M,\mathbb{G}_m) $.
	\item $ L^{*}(M,0)=\pm\chi(j_{*}M) $.
\end{enumerate}
\end{thm}
\begin{proof}
\begin{enumerate}
\item 
The ranks of
$Hom_{X}(j_{*}M,\mathbb{G}_m)$ and $\prod_{v \in S_{\infty}} H^0(K_v,M_v)/H^0(K,M)$ are the same because the regulator pairing for $j_{*}M$ is non-degenerate.
Thus, from $\cite[\mbox{I.3.4}]{Tat84}$, 
\[ \mathrm{ord}_{s=0}L(M,s)=\sum_{v \in S_{\infty}} \mathrm{rank}_{\mathbb{Z}} H^0(K_v,M_v) - \mathrm{rank}_{\mathbb{Z}}H^0(K,M)=\mathrm{rank}_{\mathbb{Z}} Hom_{X}(j_{*}M,\mathbb{G}_m) . \]
\item
Consider the two exact sequences ($\ref{ono_etale_seq1}$) and ($\ref{ono_etale_seq2}$) from proposition $\ref{ono_etale}$. Since $\mathcal{R}$ is a constructible sheaf,
by propositions $\ref{euler_piZ}$ and $\ref{euler_constructible}$, $\chi(\mathcal{R})=1$ 
and $\chi((\pi_{\lambda}')_{*}\mathbb{Z}^{m_{\lambda}})=\pm L^{*}((\pi_{\lambda}')_{*}\mathbb{Z}^{m_{\lambda}},0)$.
Hence, by theorem $\ref{mul}$ and the fact that $N$ is a finite $G_K$-module
\begin{eqnarray*}
\chi(j_{*}M)^n &=& \frac{\prod_{\lambda}\chi((\pi_{\lambda}')_{*}\mathbb{Z}^{m_{\lambda}})}
{\prod_{\mu}\chi((\pi_{\mu}')_{*}\mathbb{Z}^{m_{\mu}})}  
=  \left|\frac{\prod_{\lambda}L^{*}((\pi_{\lambda}')_{*}\mathbb{Z}^{m_{\lambda}},0)}
{\prod_{\mu}L^{*}((\pi_{\mu}')_{*}\mathbb{Z}^{m_{\mu}},0)}\right| 
=  |L^{*}(M,0)^n|.
\end{eqnarray*}
Since $L^{*}(M,0)$ is a real number, we deduce $L^{*}(M,0)=\pm\chi(j_{*}M)$.
\end{enumerate}
\end{proof}

\begin{cor}\label{value_etale_fg}
Let $T$ be an algebraic torus defined over a totally imaginary number field $K$ with character group $\hat{T}$. Then 
\[ L^{*}(\hat{T},0)=\pm\frac{[H^2_{et}(X,j_{*}\hat{T})]R(j_{*}\hat{T})}{[H^1_{et}(X,j_{*}\hat{T})][Hom_X(j_{*}\hat{T},\mathbb{G}_m)_{tor}]} 
=\pm\frac{[Ext^1_X(j_{*}\hat{T},\mathbb{G}_m)]R(j_{*}\hat{T})}{[Ext^2_X(j_{*}\hat{T},\mathbb{G}_m)][Hom_X(j_{*}\hat{T},\mathbb{G}_m)_{tor}]}.
\]
\end{cor}
\begin{proof}
The second equality follows from the first by the Artin-Verdier duality. The first equality follows from theorem $\ref{euler_value_L0}$, proposition $\ref{euler_comp1}$ and the fact that $\hat{T}$ is a torsion free abelian group.
\end{proof}

\section{Applications : Algebraic Tori}\label{chap_galois_coh}
Let $T$ be an algebraic torus defined over a totally imaginary number field $K$ with character group $\hat{T}$. Corollary $\ref{value_etale_fg}$ gives a formula for $L^{*}(\hat{T},0)$ in terms of \'etale cohomology. We shall derive another formula for $L^{*}(\hat{T},0)$ in terms of Galois cohomology and other arithmetic invariant of $T$. 

For each finite place $p$ of $K$, we write $K_p^{ur}$ for the maximal unramified extension of the completion $K_p$ of $K$. We denote by $I_p$ the inertia group of $p$. Let $O_{p}$, $O_{p}^{ur}$ and $\bar{O}_{p}$ be the valuation rings of $K_p$, $K_p^{ur}$ and $\bar{K}_p$ respectively.
\subsection{Local Galois Cohomology}
\begin{lemma}\label{lemma_localext}
Let $N$ be a discrete $G_K$-module. Let $\hat{N}=Hom_{\mathbb{Z}}(N,\bar{K}^{*})$. Then 
\begin{enumerate}
\item $Ext_X^{n}(j_{*}N,j_{*}\mathbb{G}_m) \simeq H^n(K,\hat{N})$. 
In particular,
$ Ext_X^{n}(j_{*}\hat{T},j_{*}\mathbb{G}_m) \simeq H^{n}(K,T)$.
\item Let $i:p\to X$ be a closed immersion. Then $Ext_X^{n}(j_{*}N,i_{*}\mathbb{Z}) \simeq Ext^{n}_{\hat{\mathbb{Z}}}(N^{I_p},\mathbb{Z})$.
\item For each finite prime $p$ of $K$, $H^1(I_p,N)$ is finite.
\end{enumerate}
\end{lemma}
\begin{proof}
\begin{enumerate}
\item From $\cite[\mbox{II.1.4}]{Mil06}$, $R^qj_{*}\mathbb{G}_m=0$ for $q>0$. Thus $ Ext^p_X(j_{*}N,R^qj_{*}\mathbb{G}_m) \Rightarrow Ext^{p+q}_{G_K}(N,\bar{K}^{*})$ collapses 
and yields $ Ext_X^{p}(j_{*}N,j_{*}\mathbb{G}_m) \simeq Ext^{p}_{G_K}(N,\bar{K}^{*})$. 
From $\cite[\mbox{I.0.8}]{Mil06}$, there is a spectral sequence 
$ H^p(G_K,Ext^q_{\mathbb{Z}}(N,\bar{K}^{*})) \Rightarrow Ext^{p+q}_{G_K}(N,\bar{K}^{*})$.
Since $\bar{K}^{*}$ is divisible, $Ext^q_{\mathbb{Z}}(N,\bar{K}^{*})=0$ for $q>0$.
Thus,
$H^p(G_K,Hom_{\mathbb{Z}}(N,\bar{K}^{*}))\simeq Ext^{p}_{G_K}(N,\bar{K}^{*})$. 
Hence, $Ext_X^{n}(j_{*}N,j_{*}\mathbb{G}_m)\simeq H^n(K,\hat{N})$.
Finally, if $N=\hat{T}$ then $\hat{N}=Hom_{\mathbb{Z}}(\hat{T},\bar{K}^{*}) \simeq T(\bar{K})$. Therefore, $ Ext_X^{n}(j_{*}\hat{T},j_{*}\mathbb{G}_m) \simeq H^{n}(K,T)$.

\item  From the spectral sequence 
$Ext^p_X(j_{*}N,R^qi_{*}\mathbb{Z}) \Rightarrow Ext^{p+q}_{\hat{\mathbb{Z}}}(N^{I_p},\mathbb{Z})$ and the fact that $i_{*}$ is exact, we deduce
$Ext_X^{n}(j_{*}N,i_{*}\mathbb{Z}) \simeq Ext^{n}_{\hat{\mathbb{Z}}}(N^{I_p},\mathbb{Z})$.
\item Let $L/K$ be a finite Galois extension such that $G_L$ acts trivially on $N$. 
From the Hochschild-Serre spectral sequence 
$H^r(G_{L_p^{ur}/K_p^{ur}},H^s(I_{L_p},N)) \Rightarrow H^{r+s}(I_p,N)$,
\[ 0 \to H^1(G_{L_p^{ur}/K_p^{ur}},N) \to H^1(I_p,N) \to H^0(G_{L_p^{ur}/K_p^{ur}},H^1(I_{L_p},N)) \to H^2(G_{L_p^{ur}/K_p^{ur}},N). \]
Note that $G_{L_p^{ur}/K_p^{ur}}$ is isomorphic to the inertia subgroup of $G_{L_p/K_p}$, in particular it is finite. Hence, $H^i(G_{L_p^{ur}/K_p^{ur}},N)$ is finite for $i=1,2$. Thus, it is enough to show $H^1(I_{L_p},N)$ is finite.
As $I_{L_p}$ acts trivially on $N$, it suffices to consider only two cases namely $N=\mathbb{Z}$ and $N=\mathbb{Z}/n$. Indeed, $H^1(I_{L_p},\mathbb{Z})=0$ and $H^1(I_{L_p},\mathbb{Z}/n)\simeq (O_{L_p}^{*}/(O_{L_p}^{*})^n)^D$. Hence, $H^1(I_p,N)$ is finite. 
\end{enumerate}
\end{proof}

\begin{prop}\label{local_duality_fg}
Let $N$ be a discrete $G_{K_p}$-module. Let $\hat{N}=Hom_{\mathbb{Z}}(N,\bar{K}_{p}^{*})$ and $\hat{N}^c=Hom_{\mathbb{Z}}(N,\bar{O}_{p}^{*})$. Then
\begin{enumerate}
\item $ H^0(K_p,\hat{N}^c) = 
\{ f \in Hom_{G_{K_p}}(N,\bar{K}_p^{*}) : \text{ for all } x \in H^0(K_p,N) \text{, we have } f(x) \in O_p^{*} \} $.
\item $ Ext^{2}_{\hat{\mathbb{Z}}}(N^{I_p},\mathbb{Z}) \simeq H^2(K_p,\hat{N})$.
\item The following sequence is exact
\begin{multline}\label{local_duality_sq}
0 \to H^0(K_p,\hat{N}^c) \to H^0(K_p,\hat{N}) \to Hom_{\hat{\mathbb{Z}}}(N^{I_p},\mathbb{Z}) \to
 H^0(\hat{\mathbb{Z}},H^1(I_p,N))^{D} \to \\
 \to H^1(K_p,\hat{N}) \to Ext^{1}_{\hat{\mathbb{Z}}}(N^{I_p},\mathbb{Z}) \to 0.
\end{multline}
\end{enumerate} 
\end{prop}

\begin{proof}
\begin{enumerate}
\item Let $f \in H^0(K_p,\hat{N}^c)$. For any $x \in H^0(K_p,N)$, $f(x) \in \bar{O_{p}}^{*}$ by definition. As $f$ is $G_{K_p}$-invariant, $f(x) \in O_p^{*}$. Thus, $H^0(K_p,\hat{N}^c)$ is a subset of the right hand side.

Conversely, let $f$ be an element of the right hand side. Let $L_p$ be a finite Galois extension of $K_p$ such that the Galois group $G_{L_p}$ acts trivially on $N$. For $x\in N$, $f(x) \in L_p^{*}$ as $N=H^0(L_p,N)$. We have $ N_{L_p/K_p}(f(x))=f(Tr_{L_p/K_p}(x))$.
As $Tr_{L_p/K_p}(x) \in H^0(K_p,N)$, $f(Tr_{L_p/K_p}(x)) \in O_p^{*}$. 
Hence, $N_{L_p/K_p}(f(x)) \in O_p^{*}$. We deduce that $f(x)\in O_{L_p}^{*} \subset \bar{O}_p^{*}$. As a result, the right hand side is a subset of $ H^0(K_p,\hat{N}^c)$.

\item From Tate's local duality, $H^2(K_p,\hat{N}) \simeq H^0(K_p,N)^D $ and 
$Ext^{2}_{\hat{\mathbb{Z}}}(N^{I_p},\mathbb{Z}) \simeq H^0(\hat{\mathbb{Z}},N^{I_p})^D$ which is $H^0(K_p,N)^D$. 
Thus, $H^2(K_p,\hat{N}) \simeq Ext^{2}_{\hat{\mathbb{Z}}}(N^{I_p},\mathbb{Z}) $.

\item From the spectral sequence $ H^r(\hat{\mathbb{Z}},H^s(I_p,N))\Rightarrow H^{r+s}(K_p,N)$, 
we obtain
\[0 \to H^1(\hat{\mathbb{Z}},N^{I_p}) \to H^{1}(K_p,N) \to H^0(\hat{\mathbb{Z}},H^1(I_p,N)) \to H^2(\hat{\mathbb{Z}},N^{I_p}) 
\to H^{2}(K_p,N)\]
Taking Pontryagin dual and use Tate's local duality theorem, we have
\begin{equation}\label{local_duality_sequence1}
 \widehat{H^0(K_p,\hat{N})} \xrightarrow{\hat{\Psi}} \widehat{Hom_{\hat{\mathbb{Z}}}(N^{I_p},\mathbb{Z})} \to H^0(\hat{\mathbb{Z}},H^1(I_p,N))^D 
\to H^1(K_p,\hat{N}) \to Ext^{1}_{\hat{\mathbb{Z}}}(N^{I_p},\mathbb{Z}) \to 0. 
\end{equation} 

Let $W=\mathrm{cok}\hat{\Psi}$. As $H^0(\hat{\mathbb{Z}},H^1(I_p,N))$ is finite by lemma $\ref{lemma_localext}$, $W$ is finite. To complete the proof, we shall show the following sequence is exact.
\[ 0 \to H^0(K_p,\hat{N}^c)  \to H^0(K_p,\hat{N}) \xrightarrow{\Psi} Hom_{\hat{\mathbb{Z}}}(N^{I_p},\mathbb{Z}) 
\to W \to 0.\]
The map $\Psi$ is defined as follows : for $f \in H^0(K_p,\hat{N})$ and $x \in N^{I_p}$, $\Psi(f)(x)=v(f(x))$ where $v$ is the normalized valuation of $K_p$.
Then $\Psi$ is a continuous map and 
\[ \ker\Psi=\{f \in Hom_{G_{K_p}}(N,\bar{K_p}^{*}) : \text{ for all } x \in N^{I_p} \text{, we have } f(x) \in (O_p^{ur})^{*}\}. \]

\textbf{Claim :}  $\ker \Psi = H^0(K_p,\hat{N}^c)$.

\textbf{Proof of claim :}
\begin{itemize}
\item Let $f\in H^0(K_p,\hat{N}^c)$ and $x\in N^{I_p}$. Then 
$f(x) \in H^0(I_p,\bar{K}^{*})\cap \bar{O_p}^{*}=(O_p^{ur})^{*}$.
Therefore, $H^0(K_p,\hat{N}^c) \subset \ker \Psi$.
\item To prove the other inclusion, we use the description of $H^0(K_p,\hat{N}^c)$ from part 1. 
Let $f \in \ker \Psi$ and $x \in H^0(K,N)$. Then $f(x)\in (O_p^{ur})^*$ by definition. Since $f(x)\in K_p^{*}$, $f(x) \in (O_p^{ur})^* \cap K_p^{*}=O_p^{*}$. Hence, $\ker \Psi \subset H^0(K_p,\hat{N}^c)$.
\end{itemize}

\end{enumerate}
Let $W'=\mathrm{cok}(\Psi)$. We have the following exact sequence where all the maps are strict morphisms 
$\cite[\mathrm{page 13}]{Mil06}$
\[ 0 \to H^0(K_p,\hat{N}^c)  \to H^0(K_p,\hat{N}) \xrightarrow{\Psi} Hom_{\hat{\mathbb{Z}}}(N^{I_p},\mathbb{Z}) \to W' \to 0. \]
As profinite completion is exact for sequences with strict morphisms $\cite[\mathrm{page 14}]{Mil06}$,
\[ 0 \to \widehat{H^0(K_p,\hat{N}^c)}  \to \widehat{H^0(K_p,\hat{N})} \xrightarrow{\hat{\Psi}} \widehat{Hom_{\hat{\mathbb{Z}}}(N^{I_p},\mathbb{Z})} \to \hat{W'} \to 0 .\]
As $H^0(K_p,\hat{N}^c)$ is compact and totally disconnected (topologically it is a product of finitely many copies of $O_p^{*}$), it is a profinite group. Therefore, $\widehat{H^0(K_p,\hat{N}^c)}={H^0(K_p,\hat{N}^c)}$. Moreover, $\hat{W'}=W$ which is a finite group. Hence $W'=W$. That completes the proof of the proposition.
\end{proof}

\begin{cor}\label{local_duality_examp}
Let $N=\hat{T}$ for some torus $T$ over $K_p$. Then $H^0(K_p,\hat{N})=T(K_p)$, 
	$H^0(K_p,\hat{N}^c)=T^c_p$, the maximal compact subgroup of $T(K_p)$ and 
\begin{equation}\label{local_duality_seq_2}
 0 \to \frac{T(K_p)}{T^c_p} \to Hom_{\hat{\mathbb{Z}}}(\hat{T}^{I_p},\mathbb{Z}) 
\to H^0(\hat{\mathbb{Z}},H^1(I_p,\hat{T}))^D \to H^1(K_p,T) \to Ext^{1}_{\hat{\mathbb{Z}}}(\hat{T}^{I_p},\mathbb{Z}) \to 0.
\end{equation}
\end{cor}

\subsection{A Formula For $L^{*}(\hat{T},0)$}

\begin{thm}\label{value_galois_fg}
Let $K$ be a totally imaginary number field and $T$ be an algebraic torus over $K$ with character group $\hat{T}$.
Let $h_T$, $R_T$ and $w_T$ be the class number, the regulator and the number of roots of unity of $T$.
\footnote{see $\cite{Ono61}$ for the definitions of these invariants.}
Let $\mathbb{III}^n(T)$ be the Tate-Shafarevich group. Then
\begin{equation}\label{galois_form}
L^{*}(\hat{T},0)=\pm \frac{h_TR_T}{w_T}\frac{[\mathbb{III}^1(T)]}{[H^{1}(K,\hat{T})]} \prod_{p \notin S_{\infty}}
[H^0(\hat{\mathbb{Z}},H^1(I_p,\hat{T}))].
\end{equation}
\end{thm}
\begin{proof}
From the short exact sequence of \'etale sheaves on $X=Spec(O_K)$
\[ 0\to \mathbb{G}_m \to j_{*}\mathbb{G}_m \to \coprod_{p\in X}i_{*}\mathbb{Z} \to 0 \]
we obtain the long exact sequence
\begin{equation}\label{value_galois_seq0}
 ...\to Ext_X^n(j_{*}\hat{T},\mathbb{G}_m) \to  Ext_X^n(j_{*}\hat{T},j_{*}\mathbb{G}_m) \to 
 \prod_{p \in X}Ext_X^n(j_{*}\hat{T},i_{*}\mathbb{Z}) \to ...
\end{equation}
By lemma $\ref{lemma_localext}$ and proposition $\ref{local_duality_fg}$, ($\ref{value_galois_seq0}$) can be rewritten as
\begin{equation}\label{value_galois_seq1}
...\to Ext_X^n(j_{*}\hat{T},\mathbb{G}_m) \to H^n(K,T) \to 
 \prod_{p \in X}Ext^{n}_{\hat{\mathbb{Z}}}(\hat{T}^{I_p},\mathbb{Z}) \to ... 
\end{equation}
Since $K$ is totally imaginary, for $n \geq 1$, the Tate-Shafarevich group $\mathbb{III}^n(T)$ is the kernel of the map $H^n(K,T) \to \prod_{p \in X}H^n(K_p,T)$. We split ($\ref{value_galois_seq1}$) into the following exact sequences
\begin{equation}\label{value_galois_seq2}
 0\to Hom_X(j_{*}\hat{T},\mathbb{G}_m) \to T(K) \to \prod_{p\in X} Hom_{\hat{\mathbb{Z}}}(\hat{T}^{I_p},\mathbb{Z})\to P \to 0 ,
\end{equation}
\begin{equation}\label{value_galois_seq4}
0 \to P \to Ext_X^1(j_{*}\hat{T},\mathbb{G}_m) \to Q \to 0 ,
\end{equation}
\begin{equation}\label{value_galois_seq5}
0 \to Q \to H^1(K,T) \xrightarrow{\Psi} \prod_{p \in X}Ext^{1}_{\hat{\mathbb{Z}}}(\hat{T}^{I_p},\mathbb{Z}) 
\to R \to 0 ,
\end{equation}
\begin{equation}\label{value_galois_seq6}
 0 \to R \to Ext_X^2(j_{*}\hat{T},\mathbb{G}_m) \to \mathbb{III}^{2}(T) \to 0 .
\end{equation}
Note that $P$, $Q$ and $R$ are finite groups as $Ext_X^1(j_{*}\hat{T},\mathbb{G}_m)$ and $Ext_X^2(j_{*}\hat{T},\mathbb{G}_m)$ are finite. From ($\ref{value_galois_seq4}$) and ($\ref{value_galois_seq6}$), we have
\begin{equation}\label{value_galois_seq17}
 \frac{[Ext_X^1(j_{*}\hat{T},\mathbb{G}_m)]}{[Ext_X^2(j_{*}\hat{T},\mathbb{G}_m)]}=\frac{[P][Q]}{[\mathbb{III}^2(T)][R]}.
\end{equation}
For each finite prime $p$ of $K$, we split the sequence $(\ref{local_duality_seq_2})$ into
\begin{equation}\label{value_galois_seq9}
 0 \to T(K_p)/T_p^{c} \to Hom_{\hat{\mathbb{Z}}}(\hat{T}^{I_p},\mathbb{Z}) \to S_p \to 0 ,
\end{equation}
\begin{equation}\label{value_galois_seq10}
 0 \to S_p \to H^0(\hat{\mathbb{Z}},H^1(I_p,\hat{T}))^{D} \to H^1(K_p,T) 
\to Ext^{1}_{\hat{\mathbb{Z}}}(\hat{T}^{I_p},\mathbb{Z}) \to 0 .
\end{equation}
Let $U_T$ be the group of units in $T(K)$ and $Cl(T)$ be the class group of $T$. We have the exact sequence
\begin{equation}\label{value_galois_seq3}
 0\to U_T \to T(K) \to \coprod_{p\in X}{T(K_p)/T_p^c} \to Cl(T) \to 0.
\end{equation}
From ($\ref{value_galois_seq9}$), we obtain the following commutative diagram
\[ \xymatrix{  0 \ar[r] & T(K) \ar[r] \ar[d] & T(K) \ar[r] \ar[d] & 0 \ar[r] \ar[d] & 0\\
            0 \ar[r] & \prod_{p \in X} \frac{T(K_p)}{T_p^{c}} \ar[r] & \prod_{p \in X}Hom_{\hat{\mathbb{Z}}}(\hat{T}^{I_p},\mathbb{Z}) \ar[r] & \prod_{p \in X}{S_p} \ar[r] & 0 \\}
\]
The Snake lemma combining with ($\ref{value_galois_seq2}$) and ($\ref{value_galois_seq3}$) yield $U_T \simeq Hom_X(j_{*}\hat{T},\mathbb{G}_m) $ and 
\begin{equation}\label{value_galois_seq7}
 0 \to Cl(T) \to P \to \prod_{p \in X}S_p \to 0 .
\end{equation}

In particular, $[Hom(j_{*}\hat{T},\mathbb{G}_m)_{tor}]=[U_{T,tor}]=w_T$. Also as $P$ is finite, ($\ref{value_galois_seq7}$) implies $\prod_{p \in X}S_p$ is finite and $[P]=h_T\prod_{p \in X}[S_p]$.
From ($\ref{value_galois_seq10}$), we have
\[  \xymatrix{   0 \ar[r] & 0 \ar[r] \ar[d] & H^1(K,T) \ar[d]^{\Delta} \ar[r] & H^1(K,T) \ar[r] \ar[d]^{\Psi} & 0 \\
   0 \ar[r] &\prod_p\frac{H^0(\hat{\mathbb{Z}},H^1(I_p,\hat{T}))^{D}}{S_p} \ar[r]  & \prod_pH^1(K_p,T) \ar[r]& \prod_pExt^1_{\hat{\mathbb{Z}}}(N^{I_p},\mathbb{Z}) \ar[r] &0  \\}
\]
By the Snake lemma and ($\ref{value_galois_seq5}$), we obtain
\begin{equation}\label{value_galois_seq11}
 0 \to \mathbb{III}^1(T) \to Q \to \prod_p\frac{H^0(\hat{\mathbb{Z}},H^1(I_p,\hat{T}))^{D}}{S_p} 
\to \mathrm{cok}\Delta \to R \to 0.
\end{equation}

From the generalized Poitou-Tate exact sequence $\cite[\mbox{I.4.20}]{Mil06}$, we have 
\begin{equation*}\label{value_galois_seq12}
0 \to \mathbb{III}^1(T) \to H^1(K,T) \xrightarrow{\Delta} \prod_{p}H^1(K_p,T) \to H^1(K,\hat{T})^{D} \to \mathbb{III}^2(T) \to 0.
\end{equation*}

Thus, $[\mathrm{cok}\Delta][\mathbb{III}^2(T)]=[H^1(K,\hat{T})]$. As $\mathrm{cok}\Delta$ and $\prod_pS_p$ are finite, so is $\prod_pH^0(\hat{\mathbb{Z}},H^1(I_p,\hat{T}))$. Therefore from ($\ref{value_galois_seq11}$),
\begin{equation}\label{value_galois_seq13}
[Q]= \frac{[\mathbb{III}^1(T)]\prod_p\frac{[H^0(\hat{\mathbb{Z}},H^1(I_p,\hat{T}))^{D}]}{[S_p]}[R]}{[\mathrm{cok}\Delta]}
	= \frac{[\mathbb{III}^1(T)][\mathbb{III}^2(T)][R]\prod_p[H^0(\hat{\mathbb{Z}},H^1(I_p,\hat{T}))^{D}]}{[H^1(K,\hat{T})]\prod_p[S_p]}.
\end{equation}
Putting together ($\ref{value_galois_seq17}$),($\ref{value_galois_seq7}$) and ($\ref{value_galois_seq13}$), we have
\begin{eqnarray}\label{value_galois_seq14}
 \frac{[Ext_X^1(j_{*}\hat{T},\mathbb{G}_m)]}{[Ext_X^2(j_{*}\hat{T},\mathbb{G}_m)]}
&=& \frac{h_T[\mathbb{III}^1(T)]\prod_p[H^0(\hat{\mathbb{Z}},H^1(I_p,\hat{T}))^{D}]}{[H^1(K,\hat{T})]}.
\end{eqnarray}
From corollary $\ref{value_etale_fg}$ and the fact that $R(j_{*}\hat{T})=R_T$, we finally obtain formula $(\ref{galois_form})$.
\end{proof}
\subsection{A Class Number Formula}
We can interpret equation ($\ref{value_galois_seq14}$) in the proof of theorem $\ref{value_galois_fg}$ as a formula for the class number of a torus.
\begin{prop}\label{prop_class_number_T}
Let $T$ be an algebraic torus over a totally imaginary number field $K$. Then 
\begin{eqnarray}
 h_T
&=& \frac{[Ext_X^1(j_{*}\hat{T},\mathbb{G}_m)]}{[Ext_X^2(j_{*}\hat{T},\mathbb{G}_m)]}
\frac{[H^1(K,\hat{T})]}{[\mathbb{III}^1(T)]\prod_p[H^0(\hat{\mathbb{Z}},H^1(I_p,\hat{T}))^{D}]}.
\end{eqnarray}
\end{prop}
Using ($\ref{prop_class_number_T}$) we shall deduce the following theorems of Ono  and Katayama.
\begin{thm}[$\cite{Ono87}$]\label{thm_rel_class_number}
Let $L/K$ be a Galois extension of totally imaginary number fields with Galois group $G$. 
Let $L_0$ be the maximal abelian subextension of $L$ over $K$ and $I_L$ be the group of ideles of $L$. For each finite prime $v$ of $K$, choose a prime $w$ of $L$ lying over $v$ and let $D_w$ and $I_w$ be the decomposition group and the inertia group of $w$. Let $O_w$ be the ring of integers of $L_w$.  
Let $T=R_{L/K}^{(1)}(\mathbb{G}_m)$ be the norm torus corresponding to the extension $L/K$. Then 
\begin{equation}
h_T= \frac{h_L[L_0:K][H^0_T(G,O_L^{*})]}{h_K[\ker(H^0_T(G,L^{*})\to H^0_T(G,I_L))]\prod_{v}[H^0_T(D_w,O_w^{*})]}.
\end{equation}
\end{thm}
\begin{proof}
Let $\pi:Spec(L)\to Spec(K)$. We have the exact sequence of $G_K$-modules
\begin{equation}\label{rel_class_number_eqn1}
0 \to \mathbb{Z} \to \pi_{*}\mathbb{Z} \to \hat{T} \to 0.
\end{equation}
Since $R^1j_{*}\mathbb{Z}=0$, we have the exact sequence of \'etale sheaves on $X=Spec(O_K)$
\[ 0 \to \mathbb{Z} \to \pi'_{*}\mathbb{Z} \to j_{*}\hat{T} \to 0 \]
where $\pi':Spec(O_L)\to Spec(O_K)$. The long exact sequence of Ext-groups yields  
\[
0 \to  H^0_T(G,O_L^{*}) \to 
Ext^1_{X}(j_{*}\hat{T},\mathbb{G}_m) \to Pic(O_L)\to Pic(O_K) \to Ext^2_{X}(j_{*}\hat{T},\mathbb{G}_m) \to 0.  \]
From proposition $\ref{prop_class_number_T}$, we have 
\begin{equation}
h_T= \frac{h_L[H^1(K,\hat{T})][H^0_T(G,O_L^{*})]}
{h_K[\mathbb{III}^1(T)]\prod_{p}[H^0(\hat{\mathbb{Z}},H^1(I_p,\hat{T}))]}.
\end{equation}
The exact sequence ($\ref{rel_class_number_eqn1}$) induces 
\begin{equation}\label{rel_class_number_eqn2}
0 \to \mathbb{Z} \to \mathbb{Z}[G] \to \hat{T} \to 0. 
\end{equation}
From ($\ref{rel_class_number_eqn2}$) and the fact that $H^n(G,\mathbb{Z}[G])=0$, we deduce that
\[ H^1(K,\hat{T})\simeq H^1(G,\hat{T})\simeq H^2(G,\mathbb{Z}) \simeq H^1(G,\mathbb{Q}/\mathbb{Z}).\]
Therefore, $[H^1(K,\hat{T})]=[G^{ab}]=[L_0:K]$.
The fact that $\mathbb{III}^1(T)\simeq \ker(H^0_T(G,L^{*})\to H^0_T(G,I_L))$ is proved in $\cite[6.10]{PR93}$. To complete the proof we will show that if $w$ is a prime of $L$ lying above a prime $p$ of $K$ then
\begin{equation}\label{rel_class_number_eqn3}
[H^0_T(D_w,O_w^{*})]=[H^0(\hat{\mathbb{Z}},H^1(I_p,\hat{T}))].
\end{equation}
Indeed, $H^0_T(D_w,O_w^{*})\simeq I_w^{ab}$ by local class field theory. On the other hand, $H^0(\hat{\mathbb{Z}},H^1(I_p,\hat{T}))\simeq H^0(D_w/I_w,H^1(I_w,\hat{T}))$. From sequence ($\ref{rel_class_number_eqn3}$) and the fact that $\mathbb{Z}[G]$ is an induced $I_w$-module, we have $H^1(I_w,\hat{T})\simeq H^2(I_w,\mathbb{Z})$. Consider the spectral sequence 
\[ E_2^{m,n}=H^m(D_w/I_w,H^n(I_w,\mathbb{Z}))\Rightarrow E^{m+n}=H^{m+n}(D_w,\mathbb{Z}).\]
Note that $E_2^{m,1}=0$ for all $m$ and $E_2^{m,0}=0$ for $m$ odd. Therefore, we obtain the exact sequence
\[ 0 \to (D_w/I_w)^D \to D_w^D \to H^0(D_w/I_w,H^2(I_w,\mathbb{Z})) \to 0 .\]
Hence, $H^0(D_w/I_w,H^2(I_w,\mathbb{Z}))\simeq I_w^D$. Finally,
\[ [H^0(\hat{\mathbb{Z}},H^1(I_p,\hat{T}))]=[I_w^D]=[I_w^{ab}]= [H^0_T(D_w,O_w^{*})].\]
\end{proof}
\begin{cor}\label{cyclic_rel_class_number}
In theorem $\ref{thm_rel_class_number}$, suppose further that $L/K$ is a cyclic extension. For each prime $v$ of $K$, let $e_v(L/K)$ be the ramification index of $v$ in $L$. Then 
\begin{equation}
h_T= \frac{h_L[L:K][H^0_T(G,O_L^{*})]}{h_K\prod_{v}e_v(L/K)}.
\end{equation}
\end{cor}
\begin{proof}
Since $L/K$ is cyclic, by Hasse's theorem $\ker(H^0_T(G,L^{*})\to H^0_T(G,I_L))=0$. Furthermore, $[H^0_T(D_w,O_w^{*})]=[I_w]=e_v(L/K)$. Thus, the corollary follows.
\end{proof}

\begin{thm}[$\cite{Kat91}$]\label{thm_proj_class_number}
With notations as in theorem $\ref{thm_rel_class_number}$, let ${T}'$ be the dual torus of $R_{L/K}^{(1)}(\mathbb{G}_m)$. Then 
\begin{equation}
h_{T'}= \frac{h_L[H^1_T(G,O_L^{*})]}{h_K \prod_{v}e_v(L/K)}.
\end{equation}
\end{thm}
\begin{proof}
Let $N$ be the character group of $T'$. Then $N$ satisfies the following exact sequence
\begin{equation}\label{seq_proj_class_number_1}
0 \to N \to \pi_{*}\mathbb{Z} \xrightarrow{\epsilon} \mathbb{Z}\to 0.
\end{equation}
As $R^1j_{*}(\pi_{*}\mathbb{Z})=0$, we have the exact sequence of sheaves on $X$
\begin{equation}\label{seq_proj_class_number_2}
0 \to j_{*}N \to \pi'_{*}\mathbb{Z} \to \mathbb{Z} \to R^1j_{*}N \to 0.
\end{equation}
We split ($\ref{seq_proj_class_number_2}$) into 
\begin{equation}\label{seq_proj_class_number_3a}
0 \to j_{*}N \to \pi'_{*}\mathbb{Z} \to Q \to 0.
\end{equation}
\begin{equation}\label{seq_proj_class_number_3b}
0 \to Q \to \mathbb{Z} \to R^1j_{*}N \to 0.
\end{equation}
To ease notation, let $R=R^1j_{*}N$. Then $R$ is a negligible sheaf. In particular, $Ext^n_X(R,\mathbb{G}_m)=0$ for $n=0,1$. Let $\beta$ be the map $H^0_{et}(X,Q)\to H^0_{et}(X,\mathbb{Z})$. From the long exact sequences of $Ext$-groups and cohomology groups of ($\ref{seq_proj_class_number_3b}$), we obtain the following:
\begin{eqnarray*}
Hom_X(Q,\mathbb{G}_m)& \simeq & O_K^{*},   \\ {}
 [Ext^1_X(Q,\mathbb{G}_m)]
&=& h_K [Ext^2_X(R,\mathbb{G}_m)]=
h_K \prod_{p\in X}[H^1(\hat{\mathbb{Z}},H^1(I_p,N))^D], \\ {}
[Ext^2_X(Q,\mathbb{G}_m)]&=&[H^1_{et}(X,Q)^D]=\frac{[H^0_{et}(X,R)]}{[\mathrm{cok}(\beta)]}
=\frac{\prod_{p\in X}[H^0(\hat{\mathbb{Z}},H^1(I_p,N))]}{[\mathrm{cok}(\beta)]}.
\end{eqnarray*}
Similar argument applied to the sequence ($\ref{seq_proj_class_number_3a}$) yields the following :
\begin{eqnarray*}
[Ext^2_X(j_{*}N,\mathbb{G}_m)]&=& [H^1_{et}(X,j_{*}N)^D]=[\mathrm{cok}(\alpha)],   \\ {}
[Ext^1_X(j_{*}N,\mathbb{G}_m)]&=&
\frac{h_L[S][Ext^2_X(Q,\mathbb{G}_m)]}{[Ext^1_X(Q,\mathbb{G}_m)]}
=\frac{h_L[S]}{h_K[\mathrm{cok}(\beta)]}
\end{eqnarray*}
where $\alpha$ is the map $H^0_{et}(X,\pi'_{*}\mathbb{Z})\to H^0_{et}(X,Q)$ and $S$ satisfies the exact sequence
\begin{equation}\label{seq_proj_class_number_4}
0 \to O_K^{*} \to O_L^{*} \to Hom_X(j_{*}N,\mathbb{G}_m) \to S \to 0.
\end{equation}
Note that $\beta\alpha$ is the map $H^0_{et}(X,\pi'_{*}\mathbb{Z})\to H^0_{et}(X,\mathbb{Z})$ which can be identified with the map $\epsilon :H^0(G,\mathbb{Z}[G]) \to H^0(G,\mathbb{Z})$. From the exact sequence
\begin{equation}\label{seq_proj_class_number_5}
0 \to N \to \mathbb{Z}[G] \xrightarrow{\epsilon} \mathbb{Z} \to 0
\end{equation}
we deduce that $H^1(K,N)\simeq H^1(G,N) \simeq \mathrm{cok}(\epsilon)$. As $\beta$ and $\alpha
$ are injective, $[\mathrm{cok}(\alpha)][\mathrm{cok}(\beta)]=[\mathrm{cok}(\beta\alpha)]=[H^1(K,N)]$. From proposition $\ref{prop_class_number_T}$, we have 
\[ h_{T'}=\frac{h_L[S]}{h_K[\mathbb{III}^1(T')]\prod_{p\in X}[H^0(\hat{\mathbb{Z}},H^1(I_p,N))^D]}.\]
To complete the proof of the theorem, we shall prove :
 $\mathbb{III}^1(T')=0$, $[H^0(\hat{\mathbb{Z}},H^1(I_p,N))^D]=e_p(L/K)$ and $S\simeq H^1_T(G,O_L^{*})$. 
\begin{itemize}
\item The fact that $\mathbb{III}^1(T')=0$ is proved in $\cite[\mbox{page 685}]{Kat91}$. 
\item We have $H^0(\hat{\mathbb{Z}},H^1(I_p,N))\simeq H^0(D_w/I_w,H^1(I_w,N))$ where $w$ is a prime of $L$ dividing $p$. By ($\ref{seq_proj_class_number_5}$), $H^1(I_w,N)\simeq H^0_T(I_w,\mathbb{Z})$. As $D_w/I_w$ acts trivially on $H^0_T(I_w,\mathbb{Z})$,
\[ [H^0(D_w/I_w,H^1(I_w,N))]=[H^0_T(I_w,\mathbb{Z})]=[I_w]=e_p(L/K).\]
\item From the proof of theorem $\ref{value_galois_fg}$, $Hom(j_{*}N,\mathbb{G}_m)\simeq U_{T'}=T'(O_K)$. Therefore, sequence ($\ref{seq_proj_class_number_4}$) can be identified with the sequence
\begin{equation}\label{seq_proj_class_number_6}
0 \to \mathbb{G}_m(O_K) \to R_{L/K}(\mathbb{G}_m)(O_K) \to T'(O_K) \to S \to 0 
\end{equation}
which is part of the long exact sequence of cohomology associated with 
\begin{equation}\label{seq_proj_class_number_7}
0 \to \mathbb{G}_m(O_L) \to R_{L/K}(\mathbb{G}_m)(O_L) \to T'(O_L) \to 0. 
\end{equation}
 Note that $R_{L/K}(\mathbb{G}_m)(O_L)$ is an induced $G$-module thus $H^1_T(G,R_{L/K}(\mathbb{G}_m)(O_L))=0$. Consider the long exact sequence of cohomology associated with ($\ref{seq_proj_class_number_7}$) and compare with ($\ref{seq_proj_class_number_6}$), we obtain $S\simeq H^1_T(G,O_L^{*})$.
\end{itemize} 
 \end{proof}
 \begin{rmk}
 Thereoms $\ref{thm_rel_class_number}$ and $\ref{thm_proj_class_number}$ are weaker than the original theorems of Ono and Katayama because we have to assume $K$ is totally imaginary. Moreover, Katayama also obtained formulas for $h_T$ and $h_{T'}$ when $L/K$ is not Galois, see $\cite{Kat91}$.
 \end{rmk}
\subsection{An Example \textbf{$T=Spec(\mathbb{Q}(i)[x,y]/(x^2-3y^2-1))$}}
Let $K=\mathbb{Q}(i)$ and $L=\mathbb{Q}(\zeta)$ where $\zeta=e^{2\pi i/12}$. Let $T=Spec(\mathbb{Q}(i)[x,y]/(x^2-3y^2-1))$. Then $T=R_{L/K}^{(1)}(\mathbb{G}_m)$. We want to illustrate the results of this section via the torus $T$.
\begin{lemma}\label{compute_OK_OL}
\begin{enumerate}
\item
The unit group, class number and regulator of $K=\mathbb{Q}(i)$ are given by
$O_K^{*}=\mu_K =\{ \pm 1, \pm i\}$ , $h_K=1$, $R_K=1$.
\item
The unit group, class number and regulator of $L=\mathbb{Q}(\zeta)$ are given by 
$O_L^{*}=\mu_L \oplus (1-\zeta^5)^\mathbb{Z}$
where $\mu_L=\langle \zeta \rangle$,  $h_L=1$, $R_L=\log(2+\sqrt{3})$.
Furthermore, the norm map $N_{L/K} : L^{*} \to K^{*} $ is given by $N_{L/K}(a+b\sqrt{3})=a^2-3b^2$ for $a,b \in K$.
\end{enumerate}
\end{lemma}
\begin{proof}
We only give the proof for part 2. By Dirichlet Unit Theorem, $O_L^{*}$ has rank 1. The torsion subgroup of $O_L^{*}$ is 
$\{ \zeta^n : n=0,...,11 \}$. 
As $L$ is a quadratic totally imaginary extension of $\mathbb{Q}(\sqrt{3})$, $L$ is a CM field. Therefore,
$ [O_L^{*} : \mu_L \mathbb{Z}[\sqrt{3}]^{*}] = 1,2 $ by $\cite[\mbox{4.12}]{Was97}$.
Since $\mathbb{Z}[\sqrt{3}]^{*}=\{\pm 1\}\times (2+\sqrt{3})^{\mathbb{Z}}$ and $(2+\sqrt{3})\zeta=-(1-\zeta^5)^2$, we conclude that
\[ O_L^{*} = \mu_L \oplus (1-\zeta^5)^{\mathbb{Z}}. \]
The regulator of $L$ is given by 
$ R_L=\log|(1-\zeta^5)|_{\mathbb{C}}^2=\log(2+\sqrt{3})$.
The discriminant $\Delta_L=12^2$. The Minkowski's bound is
$M_L={18}/{\pi^2}<2$. Therefore,  $h_L=1$.
Finally as $L=K(\sqrt{3})$, $N_{L/K}(a+b\sqrt{3})=a^2-3b^2$ for $a,b \in K$. 
\end{proof}
\begin{lemma}\label{global_galois_coh}
Let $G=G_{L/K}$. Then $H^n_T(G,\mu_L)\simeq \mathbb{Z}/2\mathbb{Z}$ for all $n$ and
\begin{equation}
H^n_T(G,O_L^{*}) \simeq \left\{
\begin{array}{cl}
\mathbb{Z}/2\mathbb{Z} & \mbox{for $n$ odd} \\
0 & \mbox{for $n$ even.} 
\end{array}
\right.
\end{equation}
\end{lemma}
\begin{proof}
Since $G\simeq \mathbb{Z}/2\mathbb{Z}$, $ H^n_T(G,O_L^{*})$ only depend on the parity of $n$. 
We have 
$N_{L/K}(\zeta)=-1$ and $N_{L/K}(1-\zeta^5)=i$. If $u \in O_L^{*}$ then $u=\zeta^m(1-\zeta^5)^n$ for integers $m,n$. Thus,
$N_{L/K}(u)=(-1)^mi^n=i^{2m+n}$. Let $I_G$ be the augmentation ideal of $G$. By direct calculation, $\mathrm{cok} (N_{L/K}) = 0$ and
\begin{eqnarray*}
\ker(N_{L/K})= \langle \zeta^2 \rangle \oplus (2+\sqrt{3})^{\mathbb{Z}} \quad \& \quad
I_{G}O_L^{*}= \{{\sigma(u)}/{u} : u \in O_L^{*} \}=\langle \zeta^4 \rangle \oplus (2+\sqrt{3})^{\mathbb{Z}}.
\end{eqnarray*}
As a result, $H^0_T(G_{L/K},O_L^{*})={O_K^{*}}/{N_{L/K}O_L^{*}} = 0 $ and
\begin{eqnarray*}
H^{-1}_T(G,O_L^{*})= \frac{\ker(N_{L/K})}{I_{G}O_L^{*}} = 
\frac{\langle \zeta^2 \rangle\oplus (2+\sqrt{3})^{\mathbb{Z}}}
{ \langle \zeta^4 \rangle \oplus (2+\sqrt{3})^{\mathbb{Z}}} \simeq \mathbb{Z}/2\mathbb{Z}.
\end{eqnarray*}
Similarly, we can show that $H^n_T(G,\mu_L)\simeq \mathbb{Z}/2\mathbb{Z}$ for all $n$. 
\end{proof}
\begin{prop}\label{compute_L(0)_class_number}
Let $T=R_{L/K}^{(1)}(\mathbb{G}_m)$. Then
\begin{enumerate}
\item $\mathrm{ord}_{s=0}L(\hat{T},s)=1$ and $ L^{*}(\hat{T},0) ={\log(2+\sqrt{3})}/{3}$.
\item $\prod_{p}[H^0(\hat{\mathbb{Z}},H^1(I_p,\hat{T}))]=2$.
\item $h_T=1$ and $\mathbb{III}^1(T)=0$.
\end{enumerate} 
\end{prop}
\begin{proof}
\begin{enumerate}
\item As $T=R_{L/K}^{(1)}(\mathbb{G}_m)$, $L(\hat{T},s)=\zeta_L(s)/\zeta_{K}(s)$. Thus, the first part follows from  the class number formula and lemma $\ref{compute_OK_OL}$.
\item Since $O_L=\mathbb{Z}[\mu]=O_K[\mu]$, the relative discriminant $\Delta_{L/K}=3O_K$. Therefore, $3O_K$ is the only prime of $K$ ramified in $L$. Clearly, $e_3(L/K)=2$. Hence, $\prod_{p}[H^0(\hat{\mathbb{Z}},H^1(I_p,\hat{T}))]=\prod_{p}e_p(L/K)=2$.
\item As $L/K$ is a cyclic extension, $\mathbb{III}^1(T)=0$. Then using part 2), corollary $\ref{cyclic_rel_class_number}$ and lemma $\ref{global_galois_coh}$, we deduce $h_T=1$.
\end{enumerate}

 \end{proof}
\begin{prop}\label{compute_Ext_RT}
Let $T=R_{L/K}^{(1)}(\mathbb{G}_m)$. Then  $R_T=R(j_{*}\hat{T})=2\log(2+\sqrt{3})$ and
\begin{equation}
Ext^n_X(j_{*}\hat{T},\mathbb{G}_m) \simeq \left\{
\begin{array}{cl}
 \mathbb{Z}/6\mathbb{Z} \oplus \mathbb{Z} & \mbox{$n=0$} \\
 0 & \mbox{$n=1,2$.}
\end{array}
\right.
\end{equation}
\end{prop}
\begin{proof}
Since $Pic(O_K)=Pic(O_L)=0$, the long exact sequence of Ext-groups of ($\ref{rel_class_number_eqn2}$) yields  
$Ext^2_{X}(j_{*}\hat{T},\mathbb{G}_m)=0$ and 
\[
0 \to Hom_{X}(j_{*}\hat{T},\mathbb{G}_m) \to O_L^{*} \xrightarrow{N_{L/K}} O_K^{*} \to 
Ext^1_{X}(j_{*}\hat{T},\mathbb{G}_m) \to 0.  \]
From lemma $\ref{global_galois_coh}$, $Ext^1_{X}(j_{*}\hat{T},\mathbb{G}_m)\simeq O_K^{*}/N_{L/K}O_L^{*}=0$ and $Hom_{X}(j_{*}\hat{T},\mathbb{G}_m) \simeq \ker N_{L/K} \simeq \mathbb{Z}/6\mathbb{Z} \oplus \mathbb{Z}$.
Finally, as the torsion free part of $Hom_{X}(j_{*}\hat{T},\mathbb{G}_m)$ is generated by $(2+\sqrt{3})$, 
\[ R(j_{*}\hat{T})= \log|(2+\sqrt{3})|_{\mathbb{C}}^2 = 2\log(2+\sqrt{3}).\]
\end{proof}
 
 \begin{cor}
Let $T=R_{L/K}^{(1)}(\mathbb{G}_m)$. Then $\mathrm{ord}_{s=0}L(\hat{T},s)=\mathrm{rank}_{\mathbb{Z}}Hom_X(j_{*}\hat{T},\mathbb{G}_m)$,
\[ L^{*}(\hat{T},0) = \frac{\pm [Ext^1_X(j_{*}\hat{T},\mathbb{G}_m)]R(j_{*}\hat{T})}
{[Hom_X(j_{*}\hat{T},\mathbb{G}_m)_{tor}][Ext^2_X(j_{*}\hat{T},\mathbb{G}_m)]} 
=\pm \frac{h_TR_T}{w_T}\frac{[\mathbb{III}^1(T)]}{[H^{1}(K,\hat{T})]} \prod_{p \notin S_{\infty}}
[H^0(\hat{\mathbb{Z}},H^1(I_p,\hat{T}))].\]
\end{cor}
\section{Appendix: Determinants And Torsions}
We review some results about determinants of exact sequences and orders of torsion subgroups of finitely generated abelian groups.
\subsection{Determinants Of Exact Sequences}
For $n\geq 1$, consider the following exact sequence of vector spaces over $\mathbb{R}$
\begin{equation*}
0 \to V_0 \xrightarrow{T_0} V_1 \xrightarrow{T_1} ... \xrightarrow{T_{n-1}} V_n \to 0  \quad \quad  (\mathcal{E}).
\end{equation*}
Let $B_i$ be an ordered basis for $V_i$. We want to define the determinant $\nu(\mathcal{E})$ of $(\mathcal{E})$ with respect to the bases $\{B_i\}$. We shall do so inductively.
\begin{enumerate}
\item If $n=1$, then  $\nu(\mathcal{E}):=|\det(T_0)|$ with respect to the given bases.
\item If $n=2$, suppose $B_0=\{u_i\}_{i=1}^{r}$, $B_1=\{v_i\}_{i=1}^{r+s}$ and $B_2=\{w_i\}_{i=1}^{s}$.
For $i=1,...,s$, let ${T_1}^{-1}(w_i)$ be any preimage of $w_i$ under $T_2$. We can form the following elements 
$\wedge_{i=1}^{r+s} v_i$ and $(\wedge_{i=1}^{r} T_0(u_i)) \wedge (\wedge_{i=1}^{s} T_1^{-1}(w_i))$ of $\wedge_{i=1}^{r+s} V_1$. Since $\wedge_{i=1}^{r+s} V_1$ is a 1 dimensional vector space over $\mathbb{R}$, there exists a unique \textbf{positive} real number $\delta$ such that 
\[  (\wedge_{i=1}^{r} T_0(u_i)) \wedge (\wedge_{i=1}^{s} T_1^{-1}(w_i)) = \pm \delta (\wedge_{i=1}^{r+s} v_i). \]
Note that the choice of the preimages of $w_i$  under $T_1$ does not affect 
$(\wedge_{i=1}^{r} T_0(u_i)) \wedge (\wedge_{i=1}^{s} T_1^{-1}(w_i))$.
Therefore we can define $\nu(\mathcal{E}):=\delta$.
\item If $n\geq 3$, suppose we have defined $\nu(\mathcal{E})$ for $n=N$. We want to define $\nu(\mathcal{E})$ for $n=N+1$.
Let $I$ be the image of $T_{N-1}$ and choose any basis for $I$. We split $(\mathcal{E})$ into 
\[  0 \to V_0 \xrightarrow{T_0} V_1 \xrightarrow{T_1} ... \xrightarrow{T_{N-2}} V_{N-1} \xrightarrow{T_{N-1}} I \to 0  \quad \quad  (\mathcal{E}_1),
\]
\[ 0 \to I \to V_{N} \xrightarrow{T_N} V_{N+1} \to 0 \quad \quad  (\mathcal{E}_2). \]
The determinant of ($\mathcal{E}$) defined to be
$ \nu(\mathcal{E}):=\nu(\mathcal{E}_1)\nu(\mathcal{E}_2)^{(-1)^{N-1}} $. 
Note that $\nu(\mathcal{E})$ is independent of the choice of basis for $I$.
\end{enumerate}

\begin{rmk}\label{rmk_determinant}
\begin{enumerate}
\item Let ($\mathcal{E}$) be an exact sequence of $\mathbb{R}$-vector spaces
\begin{equation*}
0 \to V_0 \xrightarrow{T_0} V_1 \xrightarrow{T_1} ... \xrightarrow{T_{n-1}} V_n \to 0  \quad \quad  (\mathcal{E}).
\end{equation*}
We split ($\mathcal{E}$) into two exact sequences ($\mathcal{E}_1$) and ($\mathcal{E}_2$) such that $\beta\alpha=T_i$. 
\begin{equation*}
0 \to V_0 \xrightarrow{T_0} V_1 \xrightarrow{T_1} ... \xrightarrow{T_{i-1}} V_i \xrightarrow{\alpha} J \to 0  \quad \quad  (\mathcal{E}_1).
\end{equation*}
\begin{equation*}
0 \to J \xrightarrow{\beta} V_{i+1} \xrightarrow{T_{i+1}} ... \xrightarrow{T_{n-1}} V_n \to 0  \quad \quad  (\mathcal{E}_2).
\end{equation*}
Then by an induction argument, we can show that 
$\nu(\mathcal{E})=\nu(\mathcal{E}_1)\nu(\mathcal{E}_2)^{(-1)^{i}}$.
\item Let ($\mathcal{E}^{*}$) be the dual sequence of ($\mathcal{E}$) and let $B_i^{*}$ be the dual basis of $B_i$. Then with respect to $\{B_i^{*}\}$ and $\{B_i\}$, 
$ \nu(\mathcal{E}^{*})=\nu(\mathcal{E})^{-1}$.
\end{enumerate}
\end{rmk}

\begin{lemma}\label{det_3x3}
Consider the following commutative diagram 
\[ \xymatrixrowsep{0.25in}\xymatrix{
      &  0 \ar[d] & 0 \ar[d] & 0 \ar[d] \\
  0 \ar[r] & A_1 \ar[d]^{\theta_1} \ar[r]^{\phi_A}  & A_2  \ar[d]^{\theta_2} \ar[r]^{\psi_A}  &   A_3 \ar[d]^{\theta_3} \ar[r] & 0
 & (\mathcal{E}_A)\\
   0 \ar[r] & B_1 \ar[d]^{\tau_1} \ar[r]^{\phi_B} &   B_2  \ar[d]^{\tau_2} \ar[r]^{\psi_B}  & B_3  \ar[d]^{\tau_3} \ar[r] &  0
 & (\mathcal{E}_B)\\
    0 \ar[r] & C_1 \ar[d] \ar[r]^{\phi_C}  &   C_2 \ar[d]  \ar[r]^{\psi_C}  &   C_3 \ar[d] \ar[r] & 0
 & (\mathcal{E}_C)     \\
  & 0 & 0 & 0 \\
 & (\mathcal{E}_1) & (\mathcal{E}_2) & (\mathcal{E}_3)}
 \]
Let {$\{a_i,b_i,c_i\}_{i=1}^3$} be bases for $\{A_i,B_i,C_i\}_{i=1}^3$. 
Then with respect to these bases
\[ \frac{\nu(\mathcal{E}_2)}{\nu(\mathcal{E}_1)\nu(\mathcal{E}_3)}=\frac{\nu(\mathcal{E}_B)}{\nu(\mathcal{E}_A)\nu(\mathcal{E}_C)}. \]
\end{lemma}
\begin{proof}
By the definition of $\nu(\mathcal{E}_B)$, we have 
\begin{equation}\label{det_3x3_eq1}
 \wedge_{i}b_2^i = \pm \nu(\mathcal{E}_B)^{-1}   (\wedge_{i}\phi_B(b_1^i))\wedge (\wedge_{i} \psi_B^{-1}(b_3^i)).
\end{equation}
Let $M:=(\wedge_{i}\phi_B\theta_1(a_1^i))$ and $N:=(\wedge_{i} \phi_B\tau_1^{-1}(c_1^i))$. By the definition of $\nu(\mathcal{E}_1)$, 
\begin{equation}\label{det_3x3_eq2}
\wedge_{i}\phi_B(b_1^i) = \pm \nu(\mathcal{E}_1)^{-1} (\wedge_{i}\phi_B\theta_1(a_1^i))\wedge (\wedge_{i} \phi_B\tau_1^{-1}(c_1^i)) = \pm \nu(\mathcal{E}_1)^{-1} M \wedge N .
\end{equation}
Let $P:=(\wedge_{i}\psi_B^{-1}\theta_3(a_3^i))$ and $Q:=(\wedge_{i} \psi_B^{-1}\tau_3^{-1}(c_3^i))$. By the definition of
$\nu(\mathcal{E}_3)$, 
\begin{equation}\label{det_3x3_eq3}
\wedge_{i}\psi_B^{-1}(b_3^i) = \pm \nu(\mathcal{E}_3)^{-1} (\wedge_{i}\psi_B^{-1}\theta_3(a_3^i))\wedge (\wedge_{i} \psi_B^{-1}\tau_3^{-1}(c_3^i)) = \pm \nu(\mathcal{E}_3)^{-1} P \wedge Q.
\end{equation}
 Putting together ($\ref{det_3x3_eq1}$), ($\ref{det_3x3_eq2}$) and ($\ref{det_3x3_eq3}$), we deduce
\begin{equation}\label{det_3x3_eq4}
\wedge_{i}b_2^i = \pm \nu(\mathcal{E}_B)^{-1}\nu(\mathcal{E}_1)^{-1}\nu(\mathcal{E}_3)^{-1}
M\wedge N \wedge P \wedge Q.
\end{equation}
Let $M':=(\wedge_{i}\theta_2\phi_A(a_1^i))$, $N':=(\wedge_{i} \tau_2^{-1}\phi_C(c_1^i)) $, 
$P':=(\wedge_{i}\theta_2\psi_A^{-1}(a_3^i)) $,
 and $Q':=(\wedge_{i} \tau_2^{-1}\psi_C^{-1}(c_3^i))$.
By a similar argument, we have 
\begin{equation}\label{det_3x3_eq5}
\wedge_{i}b_2^i = \pm \nu(\mathcal{E}_2)^{-1}\nu(\mathcal{E}_A)^{-1}\nu(\mathcal{E}_C)^{-1}
M' \wedge N' \wedge P' \wedge Q'.
\end{equation}
From ($\ref{det_3x3_eq4}$) and ($\ref{det_3x3_eq5}$), it is enough to show
\[ M\wedge N \wedge P \wedge Q = M' \wedge N' \wedge P' \wedge Q'. \]
Indeed, we have $M=M'$ since $\phi_B\theta_1=\theta_2\phi_A$. Let $x=N-N'$. As $\phi_B\tau_1^{-1}(c_1^i)-\tau_2^{-1}\phi_C(c_1^i) \in \ker\tau_2=\mathrm{im}(\theta_2)$, we deduce $x$ is a finite sum of wedge products such that each product has a factor which is an element of $\mathrm{im}(\theta_2)$.

Similarly, let $y=P-P'$. As
$\psi_B^{-1}\theta_3(a_3^i)) - \theta_2\psi_A^{-1}(a_3^i) \in (\ker\psi_B) = (\mathrm{im}\phi_B)$
, $y$ is a finite sum of wedge products such that each product has a factor belonging to $\mathrm{im}(\phi_B)$.

Since $\tau_3\psi_B=\psi_C\tau_2$, $\psi_B^{-1}\tau_3^{-1}(c_3^i)-\tau_2^{-1}\psi_C^{-1}(c_3^i)$ is an element of $\ker(\tau_3\psi_B)$. As a vector space, $\ker \tau_3\psi_B$ is spanned by $\ker\psi_B=\mathrm{im}\phi_B$ and $\psi_B^{-1}\ker(\tau_3)=\psi_B^{-1}\mathrm{im}(\theta_3)$. 
Therefore 
\[ Q-Q'=(\wedge_{i} \psi_B^{-1}\tau_3^{-1}(c_3^i))- (\wedge_{i} \tau_2^{-1}\psi_C^{-1}(c_3^i)) =z+t
\]
where $z,t$ are finite sums such that each summand of $z$ (respectively $t$) has a factor belonging to $\mathrm{im}(\phi_B)$ (respectively $\psi_B^{-1}\mathrm{im}(\theta_3)$).

{Claim:}
$M' \wedge x \wedge P'=0$,  \quad $M \wedge N \wedge y =0$, \quad $M \wedge N \wedge z =0$, \quad $P \wedge t =0$.

{Proof of claim:}
Recall that $M'=(\wedge_{i}\theta_2\phi_A(a_1^i))$ and $P'=(\wedge_{i}\theta_2\psi_A^{-1}(a_3^i))$.
It is clear that $\{\theta_2\phi_A(a_1^i),\theta_2\psi_A^{-1}(a_3^i)\}$ span $\mathrm{im}(\theta_2)$. As
each summand of $x$ has a factor belonging to $\mathrm{im}(\theta_2)$, $M' \wedge x \wedge P'=0$.
The rest of the claim can be proved in a similar fashion. Finally,
\begin{eqnarray*}
M\wedge N \wedge P \wedge Q &=& M \wedge N \wedge P \wedge (Q'+z+t) =  M \wedge N \wedge P \wedge Q' \\
&=& M \wedge N \wedge (P'+y) \wedge Q' = M \wedge N \wedge P' \wedge Q'  \\
&=& M' \wedge (N'+x) \wedge P' \wedge Q' = M' \wedge N' \wedge P' \wedge Q'.
\end{eqnarray*}
\end{proof}
The following proposition can be deduced from lemma $\ref{det_3x3}$ by an induction argument (whose proof we omit).
\begin{prop}\label{det_nxn}
Consider the following commutative diagram of $\mathbb{R}$-vector spaces
\begin{equation}\label{det_nxn_sq1}
\xymatrixrowsep{0.19in}\xymatrix{
          &  0  \ar[d]  &   0 \ar[d]  &   0 \ar[d] &   &   0 \ar[d] &     \\
0  \ar[r] & V_{0,0}  \ar[r]^{T_{0,0}} \ar[d]^{T'_{0,0}}  &   V_{0,1}  \ar[r]^{T_{0,1}} \ar[d]^{T'_{0,1}} &  V_{0,2}  \ar[r] \ar[d]^{T'_{0,2}} & \cdots \ar[r]  &   V_{0,n}  \ar[r] \ar[d]^{T'_{0,n}}  & 0  & (\mathcal{R}_0)\\
0  \ar[r] & V_{1,0}  \ar[r]^{T_{1,0}} \ar[d]  &   V_{1,1}  \ar[r]^{T_{1,1}} \ar[d] &  V_{1,2}  \ar[r] \ar[d] & \cdots \ar[r]  &   V_{1,n}  \ar[r] \ar[d]  & 0   & (\mathcal{R}_1)\\
        &   \vdots \ar[d]  &  \vdots  \ar[d] & \vdots \ar[d] & \ddots   &   \vdots \ar[d]  &    \\
0  \ar[r] & V_{m,0}  \ar[r]^{T_{m,0}} \ar[d]  &   V_{m,1}  \ar[r]^{T_{m,1}} \ar[d] &  V_{m,2}  \ar[r] \ar[d] & \cdots \ar[r]  &   V_{m,n}  \ar[r] \ar[d]  & 0   & (\mathcal{R}_m) \\
 &  0   &   0   &   0  &   &   0  &     \\
  &  (\mathcal{C}_0)   &   (\mathcal{C}_1)   &   (\mathcal{C}_{2})  &   &   (\mathcal{C}_n)  &  
}
\end{equation}
Let $B_{i,j}$ be an ordered basis for $V_{i,j}$. Then with respect to the bases $B_{i,j}$,
\begin{equation}\label{det_nxn_eq1}
\prod_{i=0}^{n}\nu(\mathcal{C}_i)^{(-1)^{i}} =  \prod_{i=0}^{m}\nu(\mathcal{R}_i)^{(-1)^{i}}.
\end{equation}
\end{prop}

\begin{cor}\label{det_2x3}
Consider the following commutative diagram with exact rows
\[\xymatrixrowsep{0.25in}\xymatrix{
  0 \ar[r] & A_1 \ar[d]^{\theta_1} \ar[r]^{\phi_A}  & A_2  \ar[d]^{\theta_2} \ar[r]^{\psi_A}  &   A_3 \ar[d]^{\theta_3} \ar[r] & 0 & (\mathcal{E}_A)\\
   0 \ar[r] & B_1  \ar[r]^{\phi_B} &   B_2  \ar[r]^{\psi_B}  & B_3   \ar[r] &  0 & (\mathcal{E}_B)}
 \]
Assume further that all the vertical maps are isomorphisms. Let $\{a_i,b_i\}_{i=1}^2$ be bases for $\{A_i,B_i\}_{i=1}^2$ respectively. Then with respect to these bases
\[ \frac{|\det\theta_1||\det\theta_3|}{|\det\theta_2|}=\frac{\nu(\mathcal{E}_A)}{\nu(\mathcal{E}_B)} .\]
\end{cor}
\subsection{Orders Of Torsion Subgroups}
For a finitely generated abelian group $M$, we write $M_f$ for $M/M_{tor}$ and by an integral basis for $M$, we mean a $\mathbb{Z}$-basis for $M_f$. Moreover, if $f:M\to N$ is a group homomorphism then $f_{tor}:M_{tor}\to N_{tor}$.
\begin{lemma}\label{torsion_group}
Consider the following exact sequence of finitely generated abelian groups
\begin{equation}\label{hom1_seq1}
0\to A \to B \xrightarrow{\phi} C \xrightarrow{\psi} D \to E \to 0.
\end{equation}
 Assume $A$ is finite. Then the orders of the torsion subgroups are related by
 \[ \frac{[A][C_{tor}]}{[B_{tor}][D_{tor}]}=\frac{1}{[\mathrm{cok}(\psi_{tor})]}. \]
\end{lemma}

\begin{proof}
This is a consequence of the fact that 
if $A$ is finite then the map $B_{tor} \to (B/A)_{tor}$ is surjective.
\end{proof}

\begin{lemma}\label{det_tor_3term}
Let ($\mathcal{E}$) be an exact sequence of finitely generated abelian groups
\[ 0 \to A \xrightarrow{\phi} B \xrightarrow{\psi} C \to 0 \] 
and $(\mathcal{E})_{\mathbb{R}}$ be the sequence ($\mathcal{E}$) tensoring with $\mathbb{R}$. Then with respect to any integral bases,
\begin{equation}\label{det_tor_3term_eq1}
\nu(\mathcal{E})_{\mathbb{R}}= \frac{[A_{tor}][C_{tor}]}{[B_{tor}]} =  [\mathrm{cok}(\psi_{tor})].
\end{equation}
\end{lemma}

\begin{proof}
From remark $\ref{rmk_determinant}$, for any section $\gamma$ of $\psi_{\mathbb{R}}$,
$ \nu(\mathcal{E})_{\mathbb{R}}=|\det\theta_{\gamma}| $ , with respect to integral bases.
As $|\det\theta_{\gamma}|$ is independent of the choice of integral bases, we only need to show that there exist a section $\gamma$ of $\psi_{\mathbb{R}}$ and integral bases of $A$, $B$ and $C$ such that ($\ref{det_tor_3term_eq1}$) holds.

Consider the following commutative diagram
\[ \xymatrixrowsep{0.2in}\xymatrix{ 0 \ar[r] & B_{tor} \ar[r] \ar[d]^{\psi_{tor}} & B \ar[d]^{\psi} \ar[r] &  B_f \ar[d]^{\psi_{f}} \ar[r] & 0 \\
                    0 \ar[r] & C_{tor} \ar[r]          & C         \ar[r] &  C_f         \ar[r] & 0 \\}
\]
The Snake lemma yields $\mathrm{cok}(\psi_f)=0$ and 
\[ 0 \to \ker\psi_{tor} \to \ker\psi=\mathrm{im}(\phi) \to \ker\psi_{f} \to 
\mathrm{cok}\psi_{tor} \to  0 .\]
Therefore, $[\mathrm{cok}\psi_{tor}]=[\ker\psi_{f}/\mathrm{im}(\phi)]$. Since $\psi_f : B_f \to C_f$ is surjective and $C_f$ is a free abelian group, there exists a section $\gamma : C_f \to B_f$ of $\psi_f$ and we have 
$ B_f = \ker(\psi_f) \oplus \gamma(C_f)$. Take any integral basis $ \{ w_i\}_{i=1}^s$ for $C_f$. 
By the Smith Normal form, there are $\mathbb{Z}$-bases $\{u_i\}_{i=1}^r$ for $A_f$ and $\{v_i\}_{i=1}^r$ for $\ker\psi_{f}$ such that $\phi_f(u_i)=m_iv_i$ where $m_i$ is a positive integer for $i=1,..,r$. Then $\{u_1,...,u_r,w_1,...,w_s\}$ 
and $\{v_1,...,v_r,\gamma(w_1),...,\gamma(w_s)\}$ form integral bases for $A_{\mathbb{R}}\oplus C_{\mathbb{R}}$ and  $B_{\mathbb{R}}$. Moreover, $[\ker\psi_{f}/\mathrm{im}(\phi)] = \prod_{i=1}^r|m_i|$.

Let $\theta_{\gamma} : A_{\mathbb{R}}\oplus C_{\mathbb{R}} \to B_{\mathbb{R}}$ be given by
$\theta(a,c)=\psi(a)+\gamma(c)$. Then with respect to the above integral bases, $\det(\theta_{\gamma})=\prod_{i=1}^r|m_i|$.
 As a result, 
\[ \nu(\mathcal{E})_{\mathbb{R}}=|\det\theta_{\gamma}| = \prod_{i=1}^r|m_i|
= {\left[\frac{\ker\psi_{f}}{\mathrm{im}(\phi)}\right]} = {[\mathrm{cok}\psi_{tor}]}= \frac{[A_{tor}][C_{tor}]}{[B_{tor}]} .\]
\end{proof}

\begin{prop}\label{det_tor}
Let $(\mathcal{E})$ be an exact sequence of finitely generated abelian groups
\[ 0 \to A_0 \xrightarrow{} A_1 \to ... \to A_n \to 0. \]
Let $(\mathcal{E})_{\mathbb{R}}$ be the sequence $(\mathcal{E})$ tensoring with $\mathbb{R}$. Let $B_i$ be an ordered integral basis for $A_i$. Then with respect to $B_i$,
\begin{equation*}\label{det_tor_eq1}
\nu(\mathcal{E})_{\mathbb{R}} = \prod_{i=0}^{n}[(A_i)_{tor}]^{(-1)^{i}} .
\end{equation*}
\end{prop}
\begin{proof}
The proof uses induction on $n$. The base case when $n=2$ is lemma $\ref{det_tor_3term}$.
\end{proof}
\begin{cor}\label{det_tor_5term}
Suppose we have an exact sequence of finitely generated abelian groups
\[ 0 \to A \to B \xrightarrow{\phi} C \xrightarrow{\psi} D \to E \to 0 \] 
where $A$ and $E$ are finite groups. Then with respect to integral bases,
\[ \nu([ 0 \to B_{\mathbb{R}} \xrightarrow{\phi} C_{\mathbb{R}} \xrightarrow{\psi} D_{\mathbb{R}} \to 0])
= \frac{[B_{tor}][D_{tor}]}{[A][C_{tor}][E]} = \frac{[\mathrm{cok}(\psi_{tor})]}{[\mathrm{cok}\psi]}. \]
\end{cor}

\bigskip
\begin{align*}
& \large \mbox{Department of Mathematics, University of Regensburg, 93053 Regensburg, Germany.} \\
& \large \mbox{Email : minh-hoang.tran@mathematik.uni-regensburg.de}
\end{align*}

	\end{document}